%% file: hmm.tex
\documentclass[11pt]{article}

\usepackage[dvipsnames]{xcolor}
\usepackage{amsmath}
\usepackage{amssymb}
\usepackage{amsthm}
\usepackage{authblk}
\usepackage{fnpct}
\usepackage{tikz,tikz-cd}
\usepackage{tikzit}
    \input{markov2.tikzstyles}
    \usetikzlibrary{decorations,external}
\usepackage[citecolor=PineGreen,colorlinks=true,linkcolor=RoyalBlue]{hyperref}

\usepackage[T1,T2A]{fontenc}
\usepackage[utf8]{inputenc}

\usepackage{enumitem}

\usepackage{bm}
\usepackage{mathrsfs,graphicx,subcaption,sidecap,mathtools,xparse,blkarray}
\usepackage[capitalize]{cleveref}

\usepackage[square,numbers]{natbib}
\usepackage{doi}

\usepackage[mathlines]{lineno}
\usepackage{etoolbox} %% <- for \pretocmd, \apptocmd and \patchcmd

\newcommand*\linenomathpatch[1]{%
  \expandafter\pretocmd\csname #1\endcsname {\linenomath}{}{}%
  \expandafter\pretocmd\csname #1*\endcsname{\linenomath}{}{}%
  \expandafter\apptocmd\csname end#1\endcsname {\endlinenomath}{}{}%
  \expandafter\apptocmd\csname end#1*\endcsname{\endlinenomath}{}{}%
}
\newcommand*\linenomathpatchAMS[1]{%
  \expandafter\pretocmd\csname #1\endcsname {\linenomathAMS}{}{}%
  \expandafter\pretocmd\csname #1*\endcsname{\linenomathAMS}{}{}%
  \expandafter\apptocmd\csname end#1\endcsname {\endlinenomath}{}{}%
  \expandafter\apptocmd\csname end#1*\endcsname{\endlinenomath}{}{}%
}
\expandafter\ifx\linenomath\linenomathWithnumbers
  \let\linenomathAMS\linenomathWithnumbers
  \patchcmd\linenomathAMS{\advance\postdisplaypenalty\linenopenalty}{}{}{}
\else
  \let\linenomathAMS\linenomathNonumbers
\fi

\let\originalleft\left
\let\originalright\right
\renewcommand{\left}{\mathopen{}\mathclose\bgroup\originalleft}
\renewcommand{\right}{\aftergroup\egroup\originalright}

\linenomathpatch{equation*}
\linenomathpatchAMS{gather}
\linenomathpatchAMS{multline}
\linenomathpatchAMS{align}
\linenomathpatchAMS{alignat}
\linenomathpatchAMS{flalign}

\allowdisplaybreaks

\usepackage{microtype} %must insert after fonts
\DeclareFontFamily{OT1}{rsfs}{}
\DeclareFontShape{OT1}{rsfs}{n}{it}{<-> rsfs10}{}
\DeclareMathAlphabet{\mathscr}{OT1}{rsfs}{n}{it}
\newcommand{\newterm}[1]{\textbf{#1}}	% new terminology introduced in bf

\newtheorem{thm}{Theorem}[section]
\newtheorem{cor}[thm]{Corollary}
\newtheorem{prop}[thm]{Proposition}

\newtheorem{lem}[thm]{Lemma}

\newtheorem{defn}[thm]{Definition}

\newtheorem{notn}[thm]{Notation}

\theoremstyle{definition}

\newtheorem{exmp}[thm]{Example}

\theoremstyle{remark}
\newtheorem{rmk}[thm]{Remark}

\Crefname{prop}{Proposition}{Propositions}
\Crefname{obs}{Observation}{Observations}
\Crefname{lem}{Lemma}{Lemmas}
\Crefname{rmk}{Remark}{Remarks}

\makeatletter
\def\cref@thmoptarg[#1]#2#3#4{%
    \ifhmode\unskip\unskip\par\fi%
    \normalfont%
    \trivlist%
    \let\thmheadnl\relax%
    \let\thm@swap\@gobble%
    \thm@notefont{\fontseries\mddefault\upshape}%
    \thm@headpunct{.}% add period after heading
    \thm@headsep 5\p@ plus\p@ minus\p@\relax%
    \thm@space@setup%
    #2% style overrides
    \@topsep \thm@preskip               % used by thm head
    \@topsepadd \thm@postskip           % used by \@endparenv
    \def\@tempa{#3}\ifx\@empty\@tempa%
      \def\@tempa{\@oparg{\@begintheorem{#4}{}}[]}%
    \else%
      \refstepcounter[#1]{#3}%  <<< cleveref modification
      \@namedef{cref@#3@alias}{#1}% added
      \def\@tempa{\@oparg{\@begintheorem{#4}{\csname the#3\endcsname}}[]}%
    \fi%
    \@tempa}%
\makeatother

\makeatletter
\def\paragraph{\@startsection{paragraph}{4}%
  \z@\z@{-\fontdimen2\font}%
  {\normalfont\it}}
\makeatother

\renewcommand{\and}{\qquad \text{and} \qquad}

\newcommand{\R}{\mathbb{R}}
\newcommand{\N}{\mathbb{N}}

\renewcommand{\Pr}{\mathbb{P}}
\newcommand{\Ch}[2]{p_{\left[#2\right]}^{#1}}
\newcommand{\filtB}[1]{\mathrm{BF}_{#1}}    % Bayes filter
\newcommand{\filtIB}[1]{\mathrm{IBF}_{#1}}  % Instantiated Bayes filter
\newcommand{\filtBP}[1]{\mathrm{BFP}_{#1}}  % Bayes filter process
\newcommand{\smoothB}[2]{\mathrm{BS}^{#1}_{#2}}    % Bayes smoother
\newcommand{\smoothIB}[3]{\mathrm{IBS}^{#1}_{#2}\left(\mathbf{#3}_{\left[#1\right]}\right)}    % Instantiated Bayes smoother
\newcommand{\filtTrans}[1]{h_{#1}}    %    ASM: feel free to change this
\newcommand{\Dist}{P}				%    ASM: feel free to change this
\newcommand{\samp}[1]{\mathsf{samp}_{#1}}    %    ASM: feel free to change this

\renewcommand{\phi}{\varphi}

\renewcommand{\d}{\mathrm{d}}

\newcommand{\monUnit}{I}
\newcommand{\catC}{\mathcal{C}}

\newcommand{\namedCat}[1]{\mathsf{#1}}
\newcommand{\id}{\mathrm{id}}

\newcommand{\finstoch}{\namedCat{FinStoch}}
\newcommand{\borelstoch}{\namedCat{BorelStoch}}

\newcommand{\finsetmulti}{\namedCat{FinSetMulti}}

\newcommand{\gauss}{\namedCat{Gauss}}
\newcommand{\stoch}{\namedCat{Stoch}}

\DeclareMathOperator{\discard}{\mathrm{del}}
\DeclareMathOperator{\Mcopy}{\mathrm{copy}}
\newcommand{\as}[1]{% 					almost surely
		\def\relstate{#1}%
		\ifx\relstate\empty
		  \text{a.s.}%
		\else
		  {#1\text{-a.s.}}%
		\fi
	}

\newcommand{\tNM}[1]{\tau_{#1}} %'transition and measure' Again, feel free to change this
\newcommand{\update}[1]{\mathrm{u}_{#1}} %Again, feel free to change this

\usepackage[nottoc]{tocbibind}

\numberwithin{equation}{section}

\usepackage{todonotes}
\author[1]{Tobias Fritz}
\author[2]{Andreas Klingler}
\author[3]{Drew McNeely}
\author[1]{\smallskip\\ Areeb Shah-Mohammed}
\author[1]{Yuwen Wang}

\affil[1]{Department of Mathematics, University of Innsbruck, Austria}
\affil[2]{Institute for Theoretical Physics, University of Innsbruck, Austria}
\affil[3]{Cockrell School of Engineering, The University of Texas at Austin, USA}
\title{\vspace{-3cm} Hidden Markov Models and the Bayes Filter\\ in Categorical Probability}

\date{}

\begin{document}

\maketitle
\thispagestyle{empty}

\begin{abstract}
	We use Markov categories to generalize the basic theory of Markov chains and hidden Markov models to an abstract setting. 
	This comprises characterizations of hidden Markov models in terms of conditional independences and algorithms for Bayesian filtering and smoothing applicable in all Markov categories with conditionals.
    	When instantiated in appropriate Markov categories, these algorithms specialize to existing ones such as the Kalman filter, forward-backward algorithm, and the Rauch--Tung--Striebel smoother.
	We also prove that the sequence of outputs of our abstract Bayes filter is itself a Markov chain with a concrete formula for its transition maps.

    	There are two main features of this categorical framework.
	The first is its abstract generality, as manifested in our unified account of hidden Markov models and algorithms for filtering and smoothing in discrete probability, Gaussian probability, measure-theoretic probability, possibilistic nondeterminism and others at the same time.
    	The second feature is the intuitive visual representation of information flow in terms of string diagrams.
\end{abstract}

\tableofcontents

\section{Introduction}

{\let\thefootnote\relax\footnotetext{\textit{Acknowledgments.} We thank Dario Stein for suggesting \Cref{prop:B_n}, Matthew B.\,Smith for help with the forward-backward algorithm, and Ramon van Handel for helpful discussion about non-linear filtering.

This research was funded in part by the Austrian Science Fund (FWF) [doi:\href{https://www.doi.org/10.55776/P35992}{10.55776/P35992}, doi:\href{https://www.doi.org/10.55776/P33122}{10.55776/P33122}, doi:\href{https://www.doi.org/10.55776/P34129}{10.55776/P34129}]. For open access purposes, the authors have applied a CC BY public copyright license to any author accepted manuscript version arising from this submission. AK further acknowledges funding of the Austrian Academy of Sciences (\"OAW) through the DOC scholarship 26547. YW's research is further supported by the University of Innsbruck Early Stage Funding Program.}}

A \newterm{hidden Markov model} is a stochastic dynamical system in discrete time, whose state is considered ``hidden'' (not directly observable), together with a sequence of noisy observations that are functions of the hidden state and each point in time.
Hidden Markov models are widely used across manifold scientific domains to model stochastic dynamical systems in which the hidden state must be approximately inferred from observations that only give partial and/or noisy information.
Making this inference can be a difficult problem in practice: while one would optimally want to use Bayesian updating, actually doing this requires conditioning on a potentially long sequence of observations.
This is problematic as it requires dealing with spaces growing exponentially with the length of the sequence, so that the na\"{\i}ve way of calculating this quickly becomes unworkable.
Recursive \newterm{filters} and \newterm{smoothers} are algorithms that address this difficulty by providing recursion formulas for this Bayesian update and, in some cases, employ suitable approximations that simplify things further.
While we will treat both filters and smoothers in the main text, let us focus on filters in this introduction.

The most widely used filter is the \newterm{Kalman filter}, which assumes the state transitions and observations of the hidden Markov model to be linear maps with additive Gaussian noise.
The most general recursive filter is the \newterm{Bayes filter}, which does not assume any kind of linearity, makes sense for discrete and continuous variables alike, but is difficult to compute with in practice.
But the world of recursive filters has a diverse landscape of variations upon these, each tailored to a different structure within an hidden Markov model or a different approximation scheme.
Some filters are designed to handle nonlinearities approximately~\cite{bar2004estimation,julier2004Unscented}.
Others employ additional structure on the state and observation spaces, such as a metric~\cite{menegaz2019unscented}.
Some may even take on a different representation of probability or information to better reflect the actual prior knowledge of the system state or improve numerical stability~\cite{bierman1982Squareroot}.
%Some variants may also incorporate several of these features.

The differences in these variations can be a lot to keep track of.
While the basic propagate-update structure is common to all of them, deriving the concrete formulas for a particular filter can be quite cumbersome.
In fact, the equations that define the Kalman filter are already quite involved and difficult to gain intuition for.
It would be helpful to have a unifying framework in which one can understand and reason about filters, and in which the common propagate-update structure can be made formal and precise.
One may then hope that this will provide a more intuitive account of filters which moreover allows one to easily determine the relevant equations in each case.

In this paper, we aim to provide such a unifying mathematical framework in terms of \newterm{categorical probability theory}.
Technically, we develop the theory of Markov chains, hidden Markov models, recursive Bayes filters and smoothers within any \newterm{Markov category with conditionals}.
Instantiating our abstract Bayes filter within any such category gives rise to a concrete filter, and we show that both the Kalman filter and the general Bayes filter arise as special cases of this construction.
We have yet to attempt to obtain other filters in this way, but we expect this to be possible for at least some of them.
Others may require further generalizations to the categorical framework such that approximations can be considered explicitly.

\subsection{Our contributions}

%\yw{I tried to shorten this section while highlighting what's new.}

We demonstrate how to implement hidden Markov models, filtering, and smoothing within the framework of categorical probability.
This unifies various filters, including the Kalman filter, the Bayes filter for finite Markov chains used in discrete probability as well as its measure-theoretic generalization, as well as a filter for nondeterministic automata (without probability).
The latter seems to be new and may have applications to information security.
The expressive language underlying our abstract approach also produces non-trivial mathematical results, such as that the sequence of filter outputs is itself a Markov chain (\Cref{prop:ChBMarkov}).

%\yw{do we need citations for the last sentence?}

The string diagram language of Markov categories also provides a visually intuitive representation of information flow within systems like hidden Markov models.

\subsection{Organization of the paper}

%We provide a summary of the main sections of the paper.
%\begin{itemize}

\Cref{sec:background} introduces relevant background from the relatively new theory of Markov categories.
This should make the paper accessible to anyone familiar with symmetric monoidal categories and string diagrams as well as basic probability.
%Much of the content of \cref{sec:background} is summarized from~\cite{fritz2019synthetic} and~\cite{fritz2023representable}, while the notation for conditionals is borrowed from~\cite{jacobs2021structured}.
In particular, we present the definition of Markov category together, with the most pertinent examples being $\finstoch$ for discrete probability, $\gauss$ for Gaussian probability, $\borelstoch$ for measure-theoretic probability and $\finsetmulti$ for possibilistic nondeterminism.
We also recall those concepts from the theory of Markov categories that are relevant to the present paper.
%\end{itemize}
In \Cref{sec:overview}, we then provide a more technical outline of the novel contributions of this paper.
In \Cref{sec:implementation}, we briefly discuss the \texttt{C++} implementation of our formalism  and highlight how the categorical structure can be facilitates modular system design.
In \Cref{sec:related_work}, we discuss related work and the connections to existing literature.

%\item 
In \Cref{sec:hmm}, we treat hidden Markov models within the framework of Markov categories with conditionals. We begin in \Cref{sec:markov_models_in_a_cat} with \newterm{Markov chains} by generalizing their traditional definition and characterization in terms of Markov properties to the categorical setting.
In \Cref{sec:hidden_markov_models_in_a_cat}, we extend this to \newterm{hidden Markov models}, which we characterize in terms of Markov properties as well (\Cref{prop:MarkovProperties}), and we provide examples in various Markov categories.

%\item 
In \Cref{sec:BayesFilter}, we develop the categorical formulation of the \newterm{Bayes filter}, which is the result of applying Bayesian inference to a hidden Markov model. 
While the na{\"\i}ve application of Bayesian inference to a sequence of observations amounts to conditioning on the entire sequence in batch, the Bayes filter is a recursive algorithm that computes this conditional recursively (\Cref{prop:B_n}).
We define the \newterm{instantiated Bayes filter} as the result of applying this Bayes filter to a fixed deterministic sequence of observations.
In \Cref{sec:examplesInstantiatedBayesFilter}, we give several examples of the instantiated Bayes filter in various Markov categories.
In $\finstoch$, it recovers the traditional Bayes filter for discrete probability.
In $\finsetmulti$, we obtain a filter for nondeterministic automata, which seems to be new and may have applications to information security.
In \Cref{ssec:BayesFilterRecursion}, we discuss the deterministic counterpart of the Bayes filter that arises when the underlying Markov category is representable, which allows us to formulate the recursion as a single string diagram.
In \Cref{sec:bayes_markov}, we show that the sequence of outputs of the Bayes filter is itself a Markov chain and give a concrete formula for its transition maps.
This illustrates how the expressive language of Markov categories can produce non-trivial mathematical results.

%\item 
In \Cref{sec:smoothing}, we develop the categorical formulation of the \newterm{Bayes smoother}, which is very similar to the Bayesian filter but addresses the more general problem of inferring the hidden state at any time $t$ from a sequence of observations up to time $n \ge t$.
We give some similar results to those obtained in \Cref{sec:BayesFilter}, notably that there is a \emph{backward in time} recursive construction of the Bayesian smoother called the \newterm{fixed-interval smoother} that also comes in both uninstantiated and instantiated flavors. We also develop a categorical generalization of the classic \newterm{forward-backward algorithm}. We then give two concrete examples of the instantiated Bayes smoother: one is to show how the forward-backward algorithm resolves to its traditional counterpart when applied in $\finstoch$, and the other example demonstrates how the fixed-interval smoother recovers the \newterm{Rauch-Tung-Striebel smoother} when applied to $\gauss$.

%	\ak{Rewrite the point below after we know what remains (probably nothing?)}
%\item 

%We showcase the abstraction by implementing the Bayes filter independently of the underlying category, so that the same algorithm can be executed in  agnostically of the underlying category. $\finstoch$ and $\finsetmulti$, written by Matvey Soloviev and the last named author, at \cite{soloviev2024markov}.

\section{Background and Summary}
\subsection{Background on Markov categories}\label{sec:background}

In the following, we summarize the basic notions of Markov categories relevant to this paper. Understanding the results and proofs requires only familiarity with the string diagrammatic calculus of Markov categories, which we sketch below. General introductions to category theory and monoidal categories can be found in \cite{basiccats, startingcats}, while \cite{baez2011rosetta, dodo, coecke2012picturing, piedeleu2025strings} provide overviews of string diagram calculus for monoidal categories in general.
In essence, a monoidal category is a category equipped with a product structure $\otimes$, which assigns a product object $A \otimes B$ to objects $A$ and $B$ and a product morphism $f \otimes g$ to morphisms $f$ and $g$. Symmetric monoidal categories also have well-behaved isomorphisms between $A \otimes B$ and $B \otimes A$. All monoidal categories discussed in this paper are symmetric.

To the best of our knowledge, the only novel contribution in this section is \Cref{prop:ConditioningCoherence}\ref{lem:det_conditional} on conditioning commuting with precomposition by deterministic morphisms.

\begin{notn}
	Throughout, we write
	\[
		\left[n\right] \coloneqq \left\{0,1,\dots,n\right\},
		\qquad \left(t,n\right] \coloneqq \left\{t+1,\dots,n\right\}
	\]
	for all $t,n \in \N$.
\end{notn}

\subsubsection{Markov categories}

The starting point of categorical probability is the notion of a \newterm{Markov category}, which extends the notion of a symmetric monoidal category by certain additional structure as follows.

%Summarizing all relevant rules into one definition gives rise to the notion of a \newterm{Markov category}.
\begin{defn}[{\cite[Definition 2.1]{fritz2019synthetic}}]
	\label{defn:MarkovCat}
	A \newterm{Markov category}\footnote{The definition was introduced under this name in~\cite[Definition 2.1]{fritz2019synthetic}, but had already appeared earlier as \newterm{affine CD category} in~\cite[Definition~2.3]{chojacobs2019strings}. We refer to~\cite[Remark~2.2]{fritz2022free} for a more detailed account of the history and even earlier closely related definitions.}
	is a semicartesian\footnote{By definition, a monoidal category is semicartesian if the monoidal unit $I$ is terminal.} symmetric monoidal category $\catC$ in which every object $X$ is equipped with a distinguished commutative comonoid structure\footnote{The relevant piece of data here is $\Mcopy_X$, as $\discard_X$ must be the unique morphism to the terminal object $I$.}
	\[
		\Mcopy_X\colon X \to X \otimes X, \qquad \discard_X \colon X \to \monUnit,
	\]
	drawn in string diagrams as
	\[
		\input{copy.tikz} \qquad \text{ and } \qquad \input{del.tikz}
	\]
	which is compatible with the tensor product in the sense that for all objects $X$ and $Y$,
	\[
		\input{MCatCopyTensorComp.tikz}
	\]
\end{defn}

\begin{rmk}
	Let us sketch the intuition behind this definition.
	The objects of a Markov category can be interpreted as spaces of values of a variable, such as the state of possible states of a system or Markov chain.
	In the Markov categories that model probability, these spaces are measurable spaces that do \emph{not} come equipped with a probability measure already.
	A morphism $f \colon A \to X$ can be thought of as a random function or process that takes an input from $A$ and produces an output in $X$ that may be random or uncertain, depicted in string diagrams like this:
	\begin{equation*}
		\input{single_output_morphism.tikz}
	\end{equation*}
	In general, composition of morphisms is depicted by using the output wire of one morphism as the input wire of another morphism.
	We think of a composite $g \circ f$ as taking an input to $f$, producing a random output, and then feeding this to $g$.
	In the various flavors of probability, a morphism $A \to X$ is a conditional probability distribution on $X$ depending on $A$, also known as a \newterm{Markov kernel}.
	These compose via the Chapman--Kolmogorov equation.
	Morphisms of the specific type $p \colon \monUnit \to X$, which effectively have trivial input, are also called \newterm{states} and drawn as triangles:
	\begin{equation*}
		\input{single_output_probability.tikz}
	\end{equation*} 
	In probability, these correspond to probability measures on $X$: the monoidal unit $\monUnit$ is a singleton, and a Markov kernel with singleton input is just a single probability measure.

	The symmetric monoidal structure formalizes the possibility of taking products of spaces, which is relevant for probability for talking about joint distributions.
	The monoidal structure on morphisms amounts to a different mode of composition, often called \newterm{parallel composition}, where $f \colon A \to X$ and $g \colon B \to Y$ compose to
	\[
		f \otimes g \colon A \otimes B \to X \otimes Y,
	\]
	and this is drawn in string diagrams by placing the diagram for $f$ next to the diagram for $g$.
	We interpret this as having the morphisms $f$ and $g$ act independently and possibly concurrently.
	For example in Markov categories modelling probability, the parallel composite of two states is the corresponding product distribution.
	A morphism from an $n$-fold tensor product to an $m$-fold tensor product is thought of as a morphism with $n$ inputs and $m$ outputs, and correspondingly drawn with $n$ input wires and $m$ output wires.

	The comonoid structures implement the idea is that the information ``flowing'' on the wires of a string diagram can be copied or discarded.
	In probability theory, composing with a discard map corresponds to marginalizing a variable, while composing with a copy map corresponds to copying the value of a variable.
	With this interpretation in mind, the equations
	\begin{equation}
		\label{eq:delBox}
		 \input{DelNatural.tikz}
	\end{equation}
	\begin{equation}
		\label{eq:delCopy}
		\input{copy_del.tikz}
	\end{equation}
	which are part of the defining conditions, reflect elementary properties of these operations in probability theory.\footnote{Intuitively, the first reflects the fact that if one marginalizes the output of a Markov kernel, then the kernel itself becomes irrelevant. The second reflects the fact that if a variable gets copied and one of the outputs gets marginalized, then the resulting map acts like an identity.}
\end{rmk}

\begin{exmp}[{See~\cite{fritz2019synthetic} for more details}]\label{exmp:MarkovCats}
  Many different flavors of probability theory can be formalized as Markov categories.
  In all of these examples, we leave it understood that the copy morphisms are the obvious ones.
\begin{enumerate}
	\item $\finstoch$ is a Markov category capturing discrete probability theory on finite sets.
      		It has finite sets as objects and stochastic matrices ${\left( f\left(y\, | \, x\right) \right)}_{x \in X, y \in Y}$ as morphisms, with matrix multiplication as composition.
		In particular, a morphism $p\colon \monUnit \to X$ is a probability distribution on a finite set $X$.
		The symmetric monoidal structure is given by the Cartesian product of sets and the tensor product (Kronecker product) of matrices.
	\item $\stoch$ is the Markov category with arbitrary measurable spaces as objects and measurable Markov kernels as morphisms.
      Composition is defined by the Chapman--Kolmogorov equation $\left(g \circ f\right)\left(T \,|\, a \right) \coloneqq \int_{X} g\left(T \, | \, x\right) \, f\left(\mathrm{d} x \, | \, a \right)$.
		The symmetric monoidal structure is given by the usual product of measurable spaces and the tensor product of Markov kernels.
		This Markov category models general measure-theoretic probability.
		In particular, a morphism $p\colon \monUnit \to X$ is just a probability measure on a measurable space $X$.
	\item $\borelstoch$ is given by $\stoch$ restricted to standard Borel spaces as objects.
		This is the Markov category that we usually use for measure-theoretic probability, since it is better-behaved than $\stoch$, and because standard Borel spaces are sufficient for most applications.
	\item $\gauss$ is the Markov category modelling Gaussian distributions and kernels. Its objects are the Euclidean spaces $\mathbb{R}^n$, morphisms $p\colon \monUnit \to \mathbb{R}^n$ are Gaussian probability measures, and general morphisms $f\colon \mathbb{R}^n \to \mathbb{R}^m$ are triples $(A, \mu, \Sigma)$ consisting of a real $m \times n$ matrix $A$, a vector $\mu \in \mathbb{R}^m$, and a positive definite $m \times m$ matrix $\Sigma$.
		Such a triple represents a stochastic map of the form
	\begin{equation*}
	 x \longmapsto A x + \mathcal{N}\left(\mu, \Sigma\right),
	\end{equation*} 
	where $\mathcal{N}\left(\mu, \Sigma\right)$ denotes a Gaussian random variable with mean $\mu$ and covariance $\Sigma$.
	Two such maps
	\[
		x \mapsto Ax + \mathcal{N}\left(\mu, \Sigma\right), \qquad y \mapsto By + \mathcal{N}\left(\nu, \Lambda\right)
	\]
	compose to
	\begin{align}
	\begin{split}
		\label{eq:gauss_composition}
		x \mapsto{} &{} \left(BA\right)x + B\mathcal{N}\left(\mu, \Sigma\right) + \mathcal{N}\left(\nu, \Lambda\right) \\
			  & = \left(BA\right)x + \mathcal{N}\left(B\mu + \nu, B \Sigma B^t + \Lambda\right),
	\end{split}
	\end{align}
	where the last equation holds as the two noise terms are assumed independent.
	For details on the symmetric monoidal structure, we refer to~\cite{fritz2019synthetic}.
  \item Besides Markov categories modeling probability distributions, there are also \emph{non-probabilistic} Markov categories.
	  For example, every cartesian monoidal category is a Markov category in which every morphism is deterministic (\Cref{defn:DetMor}).
	  In these examples, categorical probability essentially trivializes~\cite[Remark~2.4]{fritz2019synthetic}.
\item\label{exmp:MarkovCatsFinsetMulti}
	The Markov category $\finsetmulti$ models nondeterminism in the computer science sense, where the uncertainty is \emph{possibilistic} rather than probabilistic.
	It has finite sets as objects, and a morphism $f\colon X \to Y$ is given by a multivalued map, by which we mean a matrix ${\left(f\left(y \, | \, x\right)\right)}_{y \in Y, x \in X}$ with entries in $\left\{0,1\right\}$, subject to the condition that for every $x$ we have $f\left(y\, | \, x\right) = 1$ for some $y$.
	We interpret the case $f\left(y \, | \, x\right) = 1$ as saying that output $y$ is possible on input $x$, and impossible if $f\left(y \, | \, x\right) = 0$.
	Composition is given again by matrix multiplication as in $\finstoch$, but with the convention that $1 + 1 = 1$.
	Thus an output of a composite $g \circ f$ is possible on a given input if and only if there is an intermediate output $y$ which is possible on the input $x$ according to $f$, and the final output $z$ is possible on $y$ according to $g$.
	The symmetric monoidal structure is again the obvious one~\cite{fritz2019synthetic}.
\end{enumerate}
\end{exmp}

\subsubsection{Representability}

\begin{defn}[{\cite[Definition~10.1]{fritz2019synthetic}}]\label{defn:DetMor}
	A morphism $f$ in a Markov category $\catC$ is \newterm{deterministic} if it commutes with copying, that is if
  \begin{equation}
    \label{eq:DetMorDefn}
    \input{DetMorDefn.tikz}
  \end{equation}
\end{defn}

The deterministic morphisms form a cartesian monoidal subcategory denoted $\catC_{\det}$~\cite[Remark 10.13]{fritz2019synthetic}.
For example, $\finstoch_{\det}$ is\footnote{What we mean by ``is'' here is that a deterministic stochastic matrix is a matrix with exactly one non-zero entry in each row, and can therefore be identified with a genuine function, and similarly for the other cases.}
the category of finite sets and functions, $\borelstoch_{\det}$ is the category of standard Borel spaces and measurable functions, $\gauss_{\det}$ is the category of Euclidean spaces and affine maps, and $\finsetmulti_{\det}$ is again the category of finite sets and functions.

%One $\borelstoch$, one can associate to an object $X$ another object $\Dist X$, namely the space of distributions on $X$ equipped with a suitable $\sigma$-algebra, such that measurable Markov kernels $A \to X$ correspond to measurable functions $A \to \Dist X$.
%  Representability axiomatizes this phenomenon in the general setting of Markov categories, providing a sense in which objects have associated ``distribution objects'' within the same Markov category.
\begin{defn}[{\cite[Definition~3.10]{fritz2023representable}}]
	\label{defn:Representability}
	A Markov category $\catC$ is \newterm{representable} if for every object $X$, there is an 
	object $\Dist X$ and a morphism 
	$\samp{X} \colon \Dist X \to X$ such that every morphism $f \colon A \to X$ factors as
	\[
		\input{representability.tikz}
	\]
	for a unique deterministic morphism $f^{\sharp}\colon A \to \Dist X$ that we call the \newterm{deterministic counterpart} of $f$.\footnote{Equivalently, $\catC$ is representable if the inclusion functor $\catC_{\det} \rightarrowtail \catC$ has a right adjoint $\Dist\colon \catC \to \catC_{\det}$. In this formulation, the morphisms $\samp{X}$ are the counit components of this adjunction.
	We refer to~\cite{fritz2023representable} for more details e.g.~on the interaction with the monoidal structure.}
\end{defn}

For example, $\borelstoch$ is representable because a Markov kernel $A \to X$ can be identified with a measurable function $A \to \Dist X$, where $\Dist X$ is the measurable space of probability measures on $X$.
This correspondence is implemented by composition with the sampling morphism $\samp{X} \colon \Dist X \to X$, which is the Markov kernel that samples from the distribution that it receives as input.
Similarly, $\finsetmulti$ is representable because a multivalued map $A \to X$ can be identified with a function $A \to \Dist X$, where $\Dist X$ is the set of non-empty subsets of $X$.
On the other hand, $\finstoch$ is not representable.

\begin{defn}[{\cite[Definition~13.1]{fritz2023representable}\footnote{See also~\cite[Definition~5.1]{chojacobs2019strings} for an earlier version of the definition.}}]
	\label{defn:ASEquality}
	For morphisms $p \colon W \to A$ and $f,g \colon A \to X$ in a Markov category, the \newterm{almost sure equality} $f =_{\as{p}} g$ is shorthand notation for
	\[
		\input{ASEDefnEqn.tikz}
	\]
\end{defn}

Intuitively, this is thought of as $f$ and $g$ behaving the same on the support of $p$, where $p$ is itself thought of as a state depending on a parameter $A$.
In $\borelstoch$, it specializes to the usual notion of almost sure equality~\cite[Example~13.3]{fritz2019synthetic}, and therefore the same holds for its subcategories $\finstoch$ and $\gauss$.

\begin{defn}[{\cite[Definition 3.18]{fritz2023representable}}]
	\label{defn:ASCompRep}
	A representable Markov category $\catC$ is \newterm{\as{}-compatibly representable} if for all $p\colon W \to A$ and $f,g\colon A \to X$,\footnote{\label{footnote:ASCompGenImp}The converse of this implication holds for representable Markov categories in general, as can be seen by post-composing with the relevant sampling map.}
	\begin{equation}
		\label{eq:ASCompRep}
		f =_{\as{p}} g \quad \implies \quad f^\sharp =_{\as{p}} g^\sharp.
	\end{equation}
\end{defn}

Most representable Markov categories that we know of are also $\as{}$-compatibly representable, including $\borelstoch$ and $\finsetmulti$~\cite{fritz2023representable}.

%\ak{Removed so far Kolmogorov products and Parametric Markov categories. I think that both are only necessary for the comments in Section 2 and those comments are easily understood also without the categorical background of these notions.}

\subsubsection{Conditionals}

Many Markov categories satisfy additional axioms that are relevant for categorical probability theory.
The existence of conditionals in particular is a key feature, and we refer to~\cite[Example 11.3]{fritz2019synthetic} and~\cite[Definition 3.5]{chojacobs2019strings} for more details.

\begin{defn}[{\cite[Definition 11.5]{fritz2019synthetic}}]\label{defn:Conditionals}
	A morphism $f \colon A \to X \otimes Y$ in a Markov category $\catC$ has a \newterm{conditional} $f_{|X} \colon A \otimes X \to Y$ with respect to $X$ if it factorizes as\footnote{There is an obvious mirror-image definition of $f_{X|}$. Since these two definitions are interchangeable by a swap, and whichever one is used should be obvious in context throughout the paper, we only treat one of them explicitly.}
\begin{equation}\label{eq:ConditionalDefn}
\input{def_conditional.tikz}
\end{equation}
  We say that $\catC$ \newterm{has conditionals} if such a conditional always exists.
\end{defn}

\begin{exmp}\label{exmps:CatsWithConditionals}
	The main examples from~\cref{exmp:MarkovCats} all have conditionals.
  
%In the following, we discuss the existence of conditionals in the Markov categories given in~\cref{exmp:MarkovCats}.
\begin{enumerate}
	\item\label{cond:FinStoch}
		$\finstoch$ has conditionals given by
		\begin{equation}
			\label{eq:FinStochConditional}
			f_{|Y}\left(x \,| \,y, a\right) \coloneqq \frac{f\left(x,y \,|\,a\right)}{\sum_{x'} f\left(x', y \,| \,a\right)}
		\end{equation}
		whenever the denominator is nonzero, and with $f_{|Y}\left(x \,| \,y, a\right)$ an arbitrary distribution on $X$ whenever it is zero.
	\item In $\borelstoch$, a conditional $p_{|Y} \colon Y \to X$ for a probability measure $p \colon \monUnit \to X \otimes Y$ is known as a \newterm{regular conditional probability}.
      $\borelstoch$ has conditionals since such regular conditional probabilities exist, even with measurable dependence on an additional parameter from some space $A$, although the larger category $\stoch$ does not have conditionals~\cite[Example 11.3]{fritz2019synthetic}.
\item\label{cond:Gauss} $\gauss$ has conditionals, and for our upcoming treatment of the Kalman filter it is instructive to work out what they are.\footnote{See also~\cite[Section~6.5]{stein2023gaussian} for a more geometric construction of conditionals in a closely related category of \emph{extended} Gaussians.}
		A general morphism $f\colon A \to X \otimes Y$ can be written as
	\[
		\begin{pmatrix} x \\ y
		\end{pmatrix} = \begin{pmatrix}
		M \\ N
		\end{pmatrix} a + \mathcal{N}\left(v, C\right),
	\]
	where the Gaussian noise has mean $v$ and covariance matrix $C$ of the form
	\[
		v = \begin{pmatrix}
		s \\ t
		\end{pmatrix} \quad \text{ and } \quad C = \begin{pmatrix} C_{\xi\xi} & C_{\xi\eta} \\ C_{\eta\xi} & C_{\eta\eta}\end{pmatrix}.
	\]
	As per~\cite[Example 11.8]{fritz2019synthetic}, a conditional $f_{|Y} \colon A \otimes Y \to X$ is given by
	\[
		x = \left(N - C_{\xi\eta} C_{\eta\eta}^{-} M\right) a + C_{\xi\eta} C_{\eta\eta}^{-} y + \mathcal{N}(s - C_{\xi\eta} C_{\eta\eta}^{-} t, C_{\xi\xi} - C_{\xi\eta} C_{\eta\eta}^{-} C_{\eta\xi}),
	\]
	where $C_{\eta\eta}^{-}$ is the Moore--Penrose pseudoinverse of $C_{\eta\eta}$.
	 \item\label{cond:FinSetMulti} $\finsetmulti$ has conditionals~\cite[Proposition 16]{fritz2022dseparation} given simply by
	 \[
	 	f_{|Y}\left(x \,| \,y,a\right) = f\left(x, y \,| \,a\right).
	\] 
\end{enumerate}
\end{exmp}

\begin{rmk}\label{rmk:ConditionalASUniqueness}
	The conditional of a morphism is often not unique,\footnote{This is indeed already the case in $\finstoch$ due to the arbitrary choice made in the case of zero denominator in~\eqref{eq:FinStochConditional}.} and this problem in particular afflicts the Bayes filter as a particular conditional (\cref{defn:BayesFilt}).
  However, by pre-composing the defining equation~\eqref{eq:ConditionalDefn} with $\Mcopy_A$, it follows that $f_{|Y}$ is unique up to almost sure equality with respect to
	\begin{equation}
		\label{eq:ConditionalASUniqueWRT}
		\input{ConditionalASUniqueWRT.tikz}
  	\end{equation}
	Thanks to this ``almost uniqueness'', we often speak of ``the'' conditional, and hence ``the Bayes filter'', as long as it is ensured that all statements under consideration are valid regardless of the particular choice of conditional.
	This will be the case for all of our uses of conditionals in this paper.

	In an \as{}-compatibly representable Markov category, the deterministic counterpart $f_{\vert Y}^\sharp$ is clearly unique up to the same almost sure equality as well.
	This will become relevant in \Cref{ssec:BayesFilterRecursion}.
\end{rmk}

\begin{notn}\label{notn:ConditionalBentWireNotn}
	Following Jacobs~\cite[Section~7.5/6]{jacobs2021structured}, we denote a conditional also by\footnote{The subscript ``\as{}'' reminds us that conditionals are unique only up to almost sure equality as discussed in \cref{rmk:ConditionalASUniqueness}. Because of this lack of uniqueness, the dashed box cannot be viewed as an algebraic operation on morphisms.}
	\[
		\input{conditionalBendedWire.tikz}
	\]
	In this notation,\footnote{Such notation (dashed boxes and bending wires around) has also been used in the context of conditionals constructed out of ``caps'' and ``normalization'' (e.g.~\cite{tull2023activeinferencestringdiagrams}).
	While this approach leads to identical notation for conditionals, we emphasize that we do not use these notions, since this is generally not how conditionals can be constructed (e.g.~in $\borelstoch$).} the defining equation of the conditional reads as
	\[
	 	\input{ConditionalDefnEquationBentNotn.tikz}
	\]
\end{notn}

The following coherence results for conditionals admit particularly elegant depictions in terms of the dashed box notation~\cite[Exercise 7.5.4]{jacobs2021structured}.
The first two were originally proven as~\cite[Lemmas 11.11-12]{fritz2019synthetic} while the third one is new.

\begin{prop}\label{prop:ConditioningCoherence}
  In any Markov category $\catC$ with conditionals:
  \begin{enumerate}
	  \item\label{cond:CondCoherenceIter} For $f \colon A \to X \otimes Y \otimes Z$, the ``double conditional'' $(f_{|Z})_{|Y}$ is also a conditional $f_{|\left(Z \otimes Y\right)}$,
 		 \[
			\input{DoubleConditionalCoherence.tikz}
		\]
%        \[
%          \tikzfig{DoubleConditionalCoherenceASE}
%        \]
	\item\label{cond:CondCoherencePostComp} For $f \colon A \to X \otimes Y$ and $g \colon W \otimes X \to Z$, we have\footnote{This equation shows in particular that conditioning commutes with marginalization. For $f \colon A \to X \otimes Y \otimes Z$, using the marginalization morphism $g \coloneqq \id_X \otimes \discard_Y$ shows that marginalizing out $Y$ before or after conditioning on $Z$ gives the same morphism up to almost sure equality.}
  		\[
			\input{ConditioningPostCompositionCoherence.tikz}
		\]
	\item\label{lem:det_conditional}
	For $f\colon A \otimes B \to X \otimes Y$ and deterministic\footnote{Assuming\label{footnote:DetChar} $g$ to be deterministic in \ref{lem:det_conditional} is necessary, in the sense that if the conclusion holds for $A = \monUnit$ and $f = \Mcopy_B$, then $g$ is deterministic. Indeed, $\Mcopy_B$ has a conditional with respect to the second output given by $\id_B \otimes \discard_B$.
	Thus, if the equality of \Cref{prop:ConditioningCoherence}\ref{lem:det_conditional} holds for $f \coloneqq \Mcopy_B$, it is the assertion that $g\otimes \discard_B$ is a conditional of $\Mcopy \circ g$. %,
%	\[
%		\tikzfig{DetCharGCopyCond}
%	\]
%	The defining equation for that conditional is
%	\[
%		\tikzfig{DetCharGCopyCondEqnLeft}\quad = \quad \tikzfig{DetCharGCopyCondEqnMid} \quad = \quad \tikzfig{DetCharGCopyCondEqnRight} 
%	\]
  The defining equation of this conditional is precisely that $g$ is deterministic.}
 $g\colon C \to B$, we have
	\[
		\input{deterministicConditioning.tikz}
	\]
	%where the almost sure equality is with respect to
	%\[
	%	\tikzfig{deterministicConditioningASEDiag}
	%\]
\end{enumerate}
\end{prop}

\begin{proof}
	We only prove the new part~\ref{lem:det_conditional} and restrict to $A = \monUnit$, since the general case follows by using $\id_A \otimes g$ in place of $g$.
	In this case we get
	\[
		\input{conditional_f_circ_g.tikz}
	\]
	where we use the definition of the conditional in the first step and the determinism assumption in the second.
	This shows that $f_{|Y} \circ \left(g \otimes \id_Y\right)$ is a conditional of $f \circ g$ with respect to $Y$.
\end{proof}

\subsection{Detailed outline}\label{sec:overview}

Here, we give a brief overview of Bayesian filtering and the main results of this paper.
Traditionally, hidden Markov models are represented as graphical models corresponding to these directed acyclic graphs:
$$
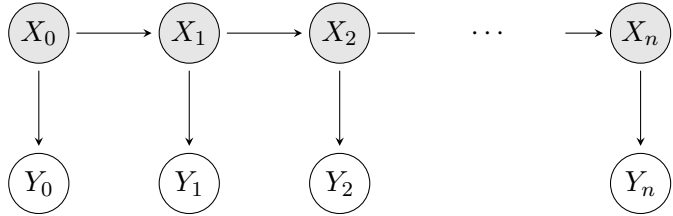
\begin{tikzpicture}
	\draw[fill=gray!20!white] (0,0) circle (0.4cm) node {$X_0$};
	\draw[fill=gray!20!white] (2,0) circle (0.4cm) node {$X_1$};
	\draw[fill=gray!20!white] (4,0) circle (0.4cm) node {$X_2$};
	\draw[fill=gray!20!white] (8,0) circle (0.4cm) node {$X_n$};

	\node at (6,0) {$\cdots$};
	
	\draw[-stealth] (0.5,0) -- (1.5,0);
	\draw[-stealth] (2.5,0) -- (3.5,0);
	\draw (4.5,0) -- (5,0);
	\draw[-stealth] (7,0) -- (7.5,0);
	
	\draw[-stealth] (0,-0.5) -- (0,-1.5);
	\draw[-stealth] (2,-0.5) -- (2,-1.5);
	\draw[-stealth] (4,-0.5) -- (4,-1.5);
	\draw[-stealth] (8,-0.5) -- (8,-1.5);
	
	\draw (0,-2) circle (0.4cm) node {$Y_0$};
	\draw (2,-2) circle (0.4cm) node {$Y_1$};
	\draw (4,-2) circle (0.4cm) node {$Y_2$};
	\draw (8,-2) circle (0.4cm) node {$Y_n$};
	
\end{tikzpicture}
$$
Hence, a hidden Markov model is a particular kind of Bayesian network.
The nodes of this graph are partitioned into ``hidden'' nodes (gray) and ``observed'' nodes (white) carrying the random variables $X_i$ and $Y_i$, respectively.
Each directed arrow indicates a possible causal dependence between these variables.
Formally, this means that a joint probability distribution $p$ is a hidden Markov model if it can be factored as
\[
	p\big(x_0, x_1, \ldots, x_n, y_0, y_1, \ldots, y_n\big) = p\left(x_0\right) \cdot p\left(y_0 \,| \,p_0\right) \cdot \prod_{t=1}^n p\big(y_t \,| \,x_t\big) \cdot  p\big(x_t \,| \,x_{t-1}\big).
\]

Let us turn to the new development of this paper.
In our string-diagrammatic setting, a \newterm{Markov chain} is represented as follows:
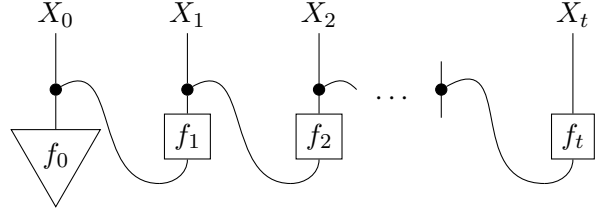
\[
	\input{mc_intro.tikz}
\]
Here, each box $f_i$ for $i > 0$ represents the transition kernel which generates an $X_i$-valued random variable as a function of the $X_{i-1}$-valued random variable.
The bending of the wires has no formal meaning in this string diagram, and we use it only in order to be able to draw the diagram horizontally.
Similarly, a \newterm{hidden Markov model} is represented as follows:
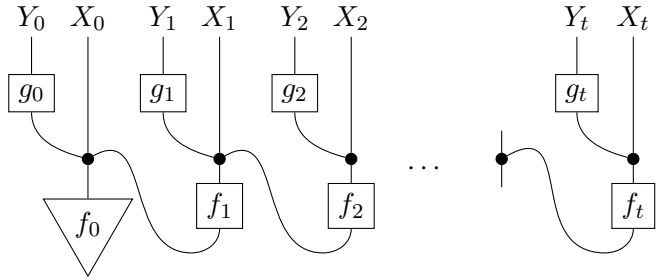
\[
	\input{hmm_intro.tikz}
\]
Now, additional $Y_i$-valued variables that depend on the $X_i$-valued ones are introduced, and this dependence is modeled by the morphisms $g_i$.
For Markov categories with conditionals, we show that Markov chains and hidden Markov models can be equivalently characterized by \newterm{Markov properties}, which are certain conditional independence relations among the different variables (\cref{prop:MarkovChainMarkovProps,prop:MarkovProperties}).

To motivate filtering and smoothing, recall that a hidden Markov model often models a situation in which one only has access to a sequence of observations $y_0, \ldots, y_t$, and one wants to guess the hidden state $x_t \in X_t$.
This is the so-called \newterm{filtering problem}.
More formally, we want to find the \newterm{Bayes filter} $\filtB{t}$, which is the conditional of the hidden variable $x_t$ given the observations $y_0, \ldots, y_t$, which by definition must be $\filtB{t}\colon Y_0 \otimes \cdots \otimes Y_t \to X_t$ such that
\[
	\input{def_bayesFilter.tikz}
\]
Here, $p_{[t]}$ is a shorthand notation for the hidden Markov model with all hidden states before time $t$ marginalized out.

Working with this definition of the Bayes filter $\filtB{t}$ in practice is hard since its domain typically grows exponentially in $t$.
This is why \emph{recursive} formulas are used in concrete applications.
In \cref{prop:B_n}, we show that $\filtB{t}$ permits such a recursive formulation in every Markov category with conditionals. However, this is still not so useful in practice, where one does not need to know the Bayes filter on \emph{all} inputs, but only on \emph{one} particular sequence of observations $y_0, \ldots, y_t$ as input.
This is why we also consider the \newterm{instantiated Bayes filter} $\filtIB{t}$, which is the morphism $\monUnit \to X_t$ obtained by plugging in a particular sequence of observations, and for which a recursion formula holds as well (\cref{prop:inst_BF}).
Subsequently, we show that our abstract $\filtIB{t}$ specializes to the Kalman filter in $\gauss$ (\cref{ssec:KalmanFilter}) and to the discrete and continuous Bayes filters in $\finstoch$ and $\borelstoch$ (\cref{ssec:instBayesFilterFinStoch,ssec:instBayesFilterBorelStoch}).
Moreover, instantiating our general construction in $\finsetmulti$ produces the \newterm{possibilistic filter}, which seems to be new.

For a hidden Markov model, we may think of the deterministic counterpart $\filtB{t}^{\sharp}$ as a function which deterministically outputs the observer's posterior distribution at time $t$ as a function of the observations (while $\filtB{t}$ samples from that distribution).
In \cref{prop:ChBMarkov}, we show that in every \as{}-compatibly representable Markov category with conditionals, the deterministic Bayes filters form a Markov chain of their own, in the sense that the posterior at time $n$, considered as a random element of $PX_{n}$, exclusively depends on the previous posterior in $PX_{n-1}$.

%% -----------------------------------------------------------------

A common question that is related to the filtering problem is how to incorporate future measurements into one's guess of a hidden state.
In other words, while the Bayes filter $\filtB{t}$ at time $t$ takes observations $y_0, \ldots, y_t$ as inputs and produces an estimate of the final hidden $x_t$, one may also want to utilize even further observations up to $y_n$ for $n > t$ in order to obtain an improved estimate of $x_t$.
This problem commonly appears in the following scenario: a Bayes filter can be used in real time as a sequence of measurements arrives in order to make a best guess of the current state using available data.
Upon reaching a time $n$, one will have obtained an estimate of the state trajectory sequence $(x_t)_{t \leq n}$, but each estimate $x_t$ will only have utilized the measurements $y_0, \ldots, y_t$ that came before it.
At the time $n$, one may want to ``refine'' this trajectory estimate, and obtain an improved sequence of estimates $(\hat{x}_t)_{t \leq n}$ where each $\hat{x}_t$ now incorporates \emph{all} measurements up to time $n$.
By nature, the $n$th state estimate will not change, that is $x_n = \hat{x}_n$, but state estimates earlier in the trajectory will typically become more accurate through this process.
This is the so-called \newterm{smoothing problem}.
In \Cref{sec:smoothing}, we define a \newterm{Bayes smoother at time $t$ with respect to time $n$} to be a morphism $\smoothB{n}{t} \colon Y_0 \otimes \cdots \otimes Y_n \to X_t$ for $t < n$ that satisfies
\[
	\input{def_bayesSmoother.tikz}
\]
where we now write $p_{[n],t}$ as shorthand for the hidden Markov model up to time $n$ with all $X_i$ \emph{except for} $X_t$ marginalized out.

Similar to $\filtB{t}$, using the definition of the Bayes smoother $\smoothB{n}{t}$ becomes difficult in practice as the marginalization and conditioning procedures become exponentially complex in $n$ and $n-t$.
Therefore, we develop several recursive procedures (\cref{obs:forward-algorithm,obs:backward-algorithm,prop:fixed-interval}) which reduce this complexity to linear.
Again, instantiated versions of these recursions tend to be more useful in practice, so we also introduce the \newterm{instantiated Bayes smoother} with corresponding recursions.
We then show how our general categorical recipes specialize to the forward-backward algorithm in $\finstoch$ (\Cref{ssec:ForwardBackwardFinstoch}) and to the Rauch--Tung--Striebel smoother in $\gauss$ (\Cref{ssec:rts-smoother}).

\subsection{Implementation}
\label{sec:implementation}

In the context of computing, the abstract nature of Markov categories allows any algorithm expressed in terms of string diagrams to be implemented in a reusable way, as it can be instantiated both in a probabilistic and possibilistic context, or generally in any Markov category that may be of interest to the programmer.
As a demonstration, Matvey Soloviev and the last named author have implemented a proof-of-concept codebase \cite{soloviev2024markov} that formalizes Markov categories as C++ \emph{abstract classes} \cite{stroustrupcpp}, which are a mechanism that enables writing programs that operate on any data structure for which a particular set of methods has been implemented. 

Based on this, we then provide an implementation of the instantiated Bayes filter (\Cref{sec:instantiatedBayesFilter}) that mirrors the string-diagrammatic formulas, together with implementations of the categories $\mathsf{FinStoch}$ and $\mathsf{FinSetMulti}$, in which the filter implementation can be used with no change to its code.
To add support for an additional category, the programmer needs only to implement a representation of objects and morphisms in that category, together with methods that implement the Markov category structure as well as conditioning.\footnote{Depending on the category, it may be advantageous to implement some frequently-used operations which can be derived from the basic ones directly, exploiting additional structure in the category for increased performance. Our implementation demonstrates an example of this for the projection morphisms $A_1\otimes \ldots \otimes A_n \rightarrow A_i$, which can be obtained as a tensor product of $n-1$ deletion morphisms with an identity, but has a faster implementation in $\mathsf{FinStoch}$ when the dimensions of the $A_i$ are known.}

\subsection{Related work}
\label{sec:related_work}

%This work takes place in the setting of Markov categories, developed by the first author in~\cite{fritz2019synthetic} building on the work of Cho and Jacobs~\cite{chojacobs2019strings}, Golubtsov~\cite{golubtsov2002kleisli}, and others; see~\cite{fritz2022free} more a more detailed account of the history.
%The traditional setting of Bayesian filtering has a rich background spurred by the works of Kalman~\cite{kalman1960filtering, kalman1961filtering}.

There are several different lines of work related to ours, in the sense that they also consider related notions of Markov chains, hidden Markov models, and/or filters in a categorical framework.

\begin{itemize}
	\item A different treatment of filters in categorical probability was developed concurrently and independently by Virgo~\cite{virgo2023Unifilar}; we refer to~\cref{rmk:Virgo} for a discussion of the precise relation.
    \item An approach to active inference for open generative models (of which hidden Markov models are an example, albeit in a different kind of category) was developed concurrently and independently by Tull et al.~\cite{tull2023activeinferencestringdiagrams}; we refer to~\cref{rmk:KST} for a discussion of the precise relation.
	\item A recent line of work by McIver et al.~is concerned with so-called \emph{quantitative information flow} in computer science~\cite{mciver2019monadic,rabehaja2019flow,alvim2020quantitative}.
		This seeks to establish information-theoretic upper bounds on how much certain inputs (such as sensitive secrets) of a process or program may influence certain outputs (such as ones visible to the general public).
		Since the hidden state of the process changes over time, it is natural to formalize this kind of situation in terms of hidden Markov models.
		One then looks for ways to quantify how much information the observed (leaked) values carry about the hidden state.

		These works are related to the present paper insofar as they employ a certain amount of categorical language, although everything takes place explicitly in the category of sets and stochastic matrices (the infinitary version of $\finstoch$).
	\item Panangaden's work on \emph{labelled Markov processes} provides a categorical framework for Markov chains with additional \emph{inputs} (called labels) rather than outputs~\cite{panangaden2009markov}.
		This work places particular emphasis on developing analogues of notions from automata theory (like bisimilarity) to labelled Markov processes.
      \item Recent work of Schauer and van der Meulen in~\cite{schauer2023compositionality} gives a categorical account of the \emph{Backward Filtering Forward Guiding (BFFG)} algorithm for smoothing, previously introduced in~\cite{miderMeulenSchauer2021smoothingOfDiffusions} and further generalized in~\cite{vandermeulen2022automatic}.
        %In the spirit of a smoothing algorithm, it diverges from the standard \emph{forward-backward algorithm} to facilitate computation where na\"{\i}ve Monte Carlo methods are intractable.
        They consider a more general setting of graphical models, of which the causal structure of a hidden Markov model is an example.
        While they do construct the algorithm using a particular Markov category, they do not use the Markov structure explicitly.
	Instead, they provide a categorical treatment in terms of optics (in the sense of~\cite{gavranovic2022spacetime}).
        They then demonstrate that compatibility of the optics construction with parallel and sequential composition allows the algorithm to be formulated as a composition of optics.
\end{itemize}
In contrast to our approach, the latter three lines of work are not concerned with a general categorical formalism, but merely use categorical language to talk about the respective structures in particular Markov categories (roughly $\finstoch$ and $\borelstoch$).
Also the transition kernels and the observation kernels are assumed to be the same at all times, whereas we allow them to vary over time.

\section{Hidden Markov models in Markov categories}
\label{sec:hmm}

In this section, we develop the notions of Markov chains and hidden Markov models in the language of categorical probability, including their characterizations in terms of Markov properties.
Instantiating these results in $\finstoch$ and $\borelstoch$ then recovers the standard discrete theory as well as the general theory of Markov chains on standard Borel spaces.
Instantiating it in $\finsetmulti$ specializes it to a theory of finite nondeterministic automata.

\subsection{Markov chains}
\label{sec:markov_models_in_a_cat}

Classically, a discrete-time Markov chain is defined as a sequence of random variables\footnote{We use lowercase notation for random variables since we exclusively reserve uppercase symbols for objects in Markov categories (typically measurable spaces serving as states spaces or observation spaces).}%\footnote{We here focus on the case where the Markov chain extends up to a fixed time $n \in \N$ only, considering the case of unbounded time in \cref{rmk:MarkovChainFinite}.}
$x_0,x_1,\dots, x_n$ taking values in a common (say discrete) \newterm{state space} $X$, such that the sequence satisfies the \newterm{Markov property}
\begin{equation}
	\label{eq:MarkovPropertyDiscrete}
	\Pr\left(x_t = \chi_t \,|\, x_{t-1} = \chi_{t-1}, \dotsc, x_0 = \chi_0\right) \,=\, \Pr\left(x_t = \chi_t \,|\, x_{t-1} = \chi_{t-1} \right)
\end{equation}
for all $1 \leq t \leq n$ and $\chi_0,\dotsc, \chi_n \in X$.\footnote{When assuming a discrete state space, this equation only needs to be considered for those sequences of values for which the left-hand side is well-defined, i.e.~for which $\Pr \left( x_{t-1} = \chi_{t-1}, \dotsc, x_0 = \chi_0 \right) > 0$.} A more abstract way to view this equation is to express it as the conditional independence
\[
	x_t \perp x_{\left[t-2\right]} \:\mid\: x_{t-1} \qquad \forall t = 1, \dotsc, n,
\]
where $x_{\left[t-2\right]}$ denotes the tuple of random variables $\left(x_0, \ldots, x_{t-2}\right)$.
Moreover, such a sequence is called \newterm{time-homogeneous} if the conditionals on the right-hand side of~\eqref{eq:MarkovPropertyDiscrete} are independent of $t$.

A well-known different definition of Markov chain specifies how the joint distribution is generated: one starts with an \newterm{initial distribution} $f_0$ corresponding to the law of $x_0$, and a \newterm{transition kernel} $f \colon X \to X$, in terms of which the joint distribution should be given by
\[
	\Pr \left( x_n = \chi_n, \dotsc, x_0 = \chi_0 \right) \,=\, f_0\left(\chi_0\right) \prod_{i=1}^n f\left(\chi_i \, | \, \chi_{i-1}\right).
\]
One then proves that the $\left(x_t\right)_{t \in \left[n\right]}$ form a time-homogeneous Markov chain if and only if there exist $f_0$ and $f$ such that the joint distribution factorizes like this.

The following result is an equivalence of this type in the setting of Markov categories.
We do not assume homogeneity in time, and we even allow the state space itself to vary in time. As is standard, we also consider two versions of the Markov property, a \emph{local} one as above and a \emph{global} one having additional conditional independences that are implied by the local ones. We refer to \cite{fritz2022dseparation} for background on conditional independence and Markov properties in the language of Markov categories.
%We refer to \cref{sec:condind} for background on conditional independence in Markov categories.

\begin{prop}
	\label{prop:MarkovChainMarkovProps}
	Let $\catC$ be a Markov category with conditionals.
	Then the following are equivalent for any $p \colon \monUnit \to \bigotimes_{t \in \left[n\right]} X_t$ in $\catC$:
	\begin{enumerate}[label= (\alph*)]
		\item\label{cond:MarkovChainMarkovPropsLoc} \newterm{Local Markov property:} $p$ displays the conditional independence\footnote{The $t = 1$ case of this condition is trivial because of $X_{[1 - 2]} = X_\emptyset = \monUnit$, so it would be enough to consider $t = 2, \dots, n$ only.}
			\[
				X_{t} \perp X_{\left[t-2\right]} \:\mid\: X_{t-1} \qquad \forall t = 1, \dotsc, n.
			\]
		\item\label{cond:MarkovChainMarkovPropsGlob} \newterm{Global Markov property:} $p$ displays the conditional independence
			\[
				X_R \perp X_T \:\mid\: X_S
			\]
			for all disjoint sets $R, S, T \subseteq \left[n\right]$ such that for every $r \in R$ and $t \in T$ there is $s \in S$ between $r$ and $t$, meaning $r < s < t$ or $r > s > t$.
		\item\label{cond:MarkovChainMarkovPropsComp}
			With the convention $X_{-1} \coloneqq \monUnit$, there is a sequence of morphisms ${\left(f_t\colon X_{t-1} \to X_t\right)}_{t=0}^n$ such that
			\begin{equation}\label{eqn:nonHomogMChExp}
				\input{p_boxState.tikz} \quad = \quad \input{p_stringDiag.tikz}
			\end{equation}
	\end{enumerate}
\end{prop}

%Writing $p_{[n]}$ for the marginal of $p$ on the first $n$ outputs, the universal property of Kolmogorov products lets us re-express this string diagram equation as the condition that for each $n$, we should have
%\begin{equation}\label{eqn:nonHomogMChnExp}
%	\tikzfig{pn_boxState} \quad = \quad \tikzfig{Ch_Xn_stringDiagNH}
%\end{equation}
In~\eqref{eqn:nonHomogMChExp}, the bending of the wires has no formal meaning, and we only need them to be able to draw the diagram horizontally.
The statement is a special case of the \newterm{categorical $d$-separation} criterion \cite[Theorem 28]{fritz2022dseparation}, which shows equivalences of this kind for arbitrary causal structures, and we thus omit the proof.

\begin{defn}
	\label{def:MarkovChain}
	When the equivalent conditions of \cref{prop:MarkovChainMarkovProps} hold, we say that $p$ is a \newterm{Markov chain} in $\catC$.
	In this case:
	\begin{enumerate}
		\item The object $X_t$ is called the \newterm{state space at time $t$}.
		\item The morphism $f_0 \colon \monUnit \to X_0$ is the \newterm{initial distribution} of the chain, and the morphisms $\left(f_t \colon X_{t-1} \to X_t\right)_{t=1}^n$ are the \newterm{transition kernels}.\footnote{As with conditionals in general (\Cref{rmk:ConditionalASUniqueness}), the transition kernels are uniquely determined by $p$ up to almost sure equality.}
		%\item The initial distribution $f_0$ is \newterm{stationary} if it is preserved by the transition kernels $f_t$, that is if $f_t \circ f_0 = f_0$ for all $t = 1, \dots, n$. 
		%	We then also call $p$ a \newterm{stationary Markov chain}.
		\item We call $p$ \newterm{time-homogeneous} if the transition kernels $f_t$ can be taken to be independent of $t$ (which requires in particular the state spaces $X_t$ to be independent of $t$).
	\end{enumerate}
\end{defn}

\begin{rmk}
	Traditionally, a Markov chain is defined over unbounded time as an \emph{infinite} sequence of random variables $\left(x_t\right)_{t \in \mathbb{N}}$, with their joint distribution forming a probability measure on an infinite product space. Extending this to the categorical setting is straightforward in terms of countable Kolmogorov products~\cite{fritzrischel2019zeroone}, but we will not consider this here.
\end{rmk}

\begin{exmp}\label{exmp:MarkovChain}
	We now describe Markov chains in the concrete Markov categories from \cref{exmp:MarkovCats}.
	\begin{enumerate}[label=(\roman*)]
		\item In $\finstoch$, a Markov chain is given by
			\begin{equation}\label{eq:finstochMC}
				p\left(x_0, \ldots, x_n\right) = f_0\left(x_0\right) \cdot \prod_{t=1}^n f_t\left(x_t \,| \,x_{t-1}\right),
			\end{equation}
			where $f_0$ is a discrete probability distribution on $X_0$, and the $f_t$ are stochastic matrices for $t \ge 1$. These $f_t$ are called \newterm{transition matrices} and form the textbook definition of Markov chains for discrete state spaces~\cite[Section 4.1]{seneta2006}.
		\item In $\borelstoch$, a Markov chain extends $\finstoch$ to standard Borel spaces $X_t$, with the joint distribution determined by
		\[
    		p(S_0 \times \dots \times S_n) = \int_{x_0 \in S_0} \dots \int_{x_n \in S_n} f_0(\mathrm{d}x_0) \cdot \prod_{t=1}^n f_t(\mathrm{d}x_t \,|\, x_{t-1})
		\]
		for all measurable $S_t \subseteq X_t$.
		\item In $\gauss$, a Markov chain is defined by a Gaussian vector $x_0 = \mathcal{N}(v_0, Q_0)$ and relations
		\[
			x_t = F_t x_{t-1} + \mathcal{N}(v_t, Q_t)
		\]
		for $t = 1, \ldots, n$, where $F_t$ is a linear map, and $v_t$ and $Q_t$ are means and covariances of independent Gaussian noise terms as in~\Cref{exmp:MarkovCats}.
		\item In $\finsetmulti$, a Markov chain follows \eqref{eq:finstochMC}, but with $f_0 \colon X_0 \to \{0,1\}$ as an initial possibility distribution and transition matrices $f_t \colon X_{t-1} \times X_t \to \{0,1\}$. For each $x_t \in X_t$, there exists at least one $x_{t+1} \in X_{t+1}$ such that $f_{t+1}(x_{t+1} \,|\, x_t) = 1$.
			Thus $f_0$ specifies allowed initial states and $f_t$ defines allowed transitions. The joint distribution $p \colon \monUnit \to \bigotimes_{t=0}^n X_t$ gives the set of possible trajectories $(x_0, \ldots, x_n)$.
		
		A time-homogeneous Markov chain in $\finsetmulti$ is therefore the same thing as a \newterm{nondeterministic finite automaton} with trivial input and at least one outgoing transition per state.
	\end{enumerate}
\end{exmp}

\begin{rmk}\label{rmk:ParamMChain}
	Given a Markov category $\catC$ and an object $W$ of $\catC$, there is a \newterm{parametric Markov category} $\catC_W$ associated to it~\cite[Section 2.2]{fritz2023representable}.
	$\catC_W$ has the same objects as $\catC$, but the morphisms have an extra dependence on a parameter represented by $W$: a morphism $X\to Y$ in $\catC_W$ is a morphism $W \otimes X \to Y$ in $\catC$.\footnote{Equivalently, $\catC_W$ is the co-Kleisli category of the so-called \emph{reader comonad} $W\otimes -$ on $\catC$.}
	It is also known that if $\catC$ has conditionals, then so does $\catC_W$ \cite[Lemma 2.10]{fritz2023representable}.
	A Markov chain in $\catC_W$ looks like
	\begin{equation}
		\label{eqn:Ch_Xn_stringDiagParam}
		\input{Ch_Xn_stringDiagParam.tikz}
	\end{equation}
	This illustrates again that $W$ can be thought of as an ``external'' parameter that is fed into each of the transition kernels.
	Since all of our results apply to parametric Markov categories as well, we automatically obtain the corresponding results with additional (measurable) dependence on a parameter.
\end{rmk}

\subsection{Hidden Markov models}
\label{sec:hidden_markov_models_in_a_cat}

We now consider \newterm{hidden Markov models}, which consist of a Markov chain of hidden states and a sequence of observations $g_t \colon X_t \to Y_t$.
The observations by themselves do typically not form a Markov chain, as they lack the local Markov property: knowing an earlier observation can refine predictions about future observations, even given the present one.

\begin{notn}
	For families of objects $\left(X_t\right)_{t \in \left[n\right]}$ and $\left(Y_t\right)_{t \in \left[n\right]}$, we write:
	\begin{enumerate}
		\item $\left(X \otimes Y\right)_{[t]}$ as shorthand for $\bigotimes_{i \in [t]} \left(X_i \otimes Y_i\right)$.
	\end{enumerate}
	Given a state $p \colon \monUnit \to \left(X \otimes Y\right)_{\left[n\right]}$, we write:
	\begin{enumerate}[resume]
		\item $p_{\left[t\right]}$ for the marginal of $p$ up to time $t$, i.e.~the marginal on $\left(X \otimes Y\right)_{\left[t\right]}$.
		\item $p^X$ for the marginal of $p$ on $X_{\left[n\right]}$, and similarly $p^Y$ for the marginal of $p$ on $Y_{\left[n\right]}$.
	\end{enumerate}
	Both notations can also be combined to $p_{\left[t\right]}^X$ and $p_{\left[t\right]}^Y$.
\end{notn}

\begin{prop}\label{prop:MarkovProperties}
	Let $\catC$ be a Markov category with conditionals with two families of objects $\left(X_t\right)_{t \in \left[n\right]}$ and $\left(Y_t\right)_{t \in \left[n\right]}$.
	Then the following are equivalent for any state $p \colon \monUnit \to \left(X \otimes Y\right)_{\left[n\right]}$:
	\begin{enumerate}[label= (\roman*)]
		\item\label{cond:backward}
			The \newterm{backward Markov property}: 
			\begin{enumerate}[label= (\alph*)]
				\item\label{cond:mark1} $X_t \perp X_{\left[t-2\right]}, Y_{\left[t-1\right]} \,| \,X_{t-1} \qquad \forall t = 1, \dots, n$,
				\item\label{cond:mark2} $Y_t \perp Y_{\left[t-1\right]}, X_{\left[t-1\right]} \,| \,X_t \hspace{36pt} \forall t = 1, \dots, n$.
			\end{enumerate}
		\item\label{cond:local}
			The \newterm{local Markov property}:
			\begin{enumerate}[label= (\alph*)]
				\item\label{cond:localA} $X_t \perp X_{\left[t-2\right]}, Y_{\left[t-1\right]} \,| \,X_{t-1} \qquad \forall t = 1, \dots, n$,
				\item\label{cond:localB} $Y_t \perp X_{\left[n\right] \setminus \left\{t\right\}}, Y_{\left[n\right] \setminus \left\{t\right\}} \,| \,X_t \qquad \forall t = 0, \dots, n$.
			\end{enumerate}
		\item\label{cond:global}
			The \newterm{global Markov property}: for all disjoint sets $R, S, T \subseteq \left[n\right]$ and disjoint sets $U, V, W \subseteq \left[n\right]$ such that
			\[
				\forall r \in R, \, t \in T \;\, \exists s \in S \colon s \text{ is between } r \text{ and } t
			\]
			and
			\[ 
				\forall u \in U, \, w \in W \;\, \exists s \in S \colon s \text{ is non-strictly between } u \text{ and } w
			\]
			we have
			\[
				X_R, Y_U \perp X_T, Y_W \,| \,X_S, Y_V.
			\]
		\item\label{cond:rep}
			There exist a state $f_0 \colon \monUnit \to X_0$ and sequences of morphisms
			\[
				f_t \colon X_{t-1} \to X_t, \qquad g_t \colon X_t \to Y_t, \qquad t = 0, \dots, n
			\]
			such that
			\begin{equation}
				\label{eqn:Ch_XY_stringDiag}
				\input{p_XY_boxState.tikz} \qquad = \qquad \input{Ch_XY_stringDiag.tikz}
			\end{equation}
	\end{enumerate}
\end{prop}

These Markov properties are rarely considered in the traditional hidden Markov model literature.
The backward Markov property is noted in~\cite{lambert2018filter}, and two special cases of the global Markov property appear in~\cite[Corollary~2.2.5]{cappe2005inference}.
However, we have found no references proving the above equivalences in measure-theoretic probability.
Our approach delivers on this simply by instantiating the proposition in $\borelstoch$.

%\begin{notn}
%	Similar to our earlier notation $\Ch{X}{n}$ and $\Ch{X}{}$, we denote by $\Ch{Y}{n}$ and $\Ch{Y}{}$ the joint distributions obtained by applying $g_n$ to the output $Y_n$ for every $n$. Explicitly
%	\[ \tikzfig{Ch_Y_boxState} \quad = \quad \tikzfig{Ch_Y_stringDiag} \] and similarly for $\Ch{Y}{}$.
	%    \[
	%    \tikzfig{Ch_Y_boxState} = \tikzfig{Ch_Y_stringDiag} %ASM: Maybe add this if helpful?
	%    \]
%\end{notn}

\begin{proof}
	The equivalence \ref{cond:local} $\Longleftrightarrow$ \ref{cond:global} $\Longleftrightarrow$ \ref{cond:rep} is precisely \cite[Theorem 34]{fritz2022dseparation} applied to the causal structure displayed in~\eqref{eqn:Ch_XY_stringDiag}.
	
	To show that the local Markov property~\ref{cond:local} implies the backward Markov property~\ref{cond:backward}, note that Condition~\ref{cond:mark1} is shared by both. Condition~\ref{cond:mark2} of the backward Markov property follows from Condition~\ref{cond:localB} of the local Markov property by marginalizing out $X_{\left[n\right] \setminus \left[t\right]}$ and $Y_{\left[n\right] \setminus \left[t\right]}$, which preserves conditional independence~\cite[Lemma~12.5(b)]{fritz2019synthetic}.

	The proof that the backward Markov property~\ref{cond:backward} is sufficient for $p$ to have the form~\eqref{eqn:Ch_XY_stringDiag} is analogous to the proof of \cite[Theorem 34]{fritz2022dseparation} (iii) $\Longrightarrow$ (i).
\end{proof}

\begin{defn}\label{defn:HiddenMarkovModel}
	When the equivalent conditions of \cref{prop:MarkovProperties} hold, we say that $p$ is a \newterm{hidden Markov model} in $\catC$.
	In this case:
	\begin{enumerate}
		\item The object $Y_t$ is called the \newterm{observation space at time $t$}.
		\item The morphisms $\left(g_t \colon X_t \to Y_t\right)_{t = 0, \dots, n}$ are the \newterm{observation kernels}.
		\item We call $p$ \newterm{time-homogeneous} if both the transition kernels $f_t$ and the observation kernels $g_t$ can be taken to be independent of $t$ (which requires, in particular, the state spaces $X_t$ and observation spaces $Y_t$ to be independent of $t$).
	\end{enumerate}
\end{defn}

So, most concretely, we can say that a hidden Markov model in $\catC$ is specified by a Markov chain ${\left(f_t \colon X_{t-1} \to X_t\right)}_{t=1}^n$ with initial distribution $f_0 \colon \monUnit \to X_0$ and an additional sequence of morphisms $\left(g_t\colon X_t \to Y_t\right)_{t=0}^n$.

\begin{exmp}\label{exmp:HMM}
We describe hidden Markov models in the Markov categories from \cref{exmp:MarkovCats}.

\begin{enumerate}[label=(\roman*)]
    \item In $\finstoch$, a hidden Markov model is described by
    \begin{equation}\label{eq:finstochHMM}
        p\left(x_0, y_0, \ldots, x_n, y_n\right) = \prod_{t=0}^n f_t\left(x_t \,|\, x_{t-1}\right) \cdot g_t\left(y_t \,|\, x_t\right),
    \end{equation}
    where $x_{-1}$ is the sole element of $X_{-1} \coloneqq \monUnit$, and the $f_t$ and $g_t$ are stochastic matrices. The $g_t$ are also called \newterm{emission probabilities}.
	\item In $\borelstoch$, hidden Markov models generalize those in $\finstoch$, allowing arbitrary standard Borel spaces for states and observations, and arbitrary measurable transition and observation kernels. Despite this generality, \cref{prop:MarkovProperties} applies and proves the characterization in terms of Markov properties in a way that completely avoids measure theory.\footnote{Measure theory is needed only for the proof that $\borelstoch$ is a Markov category with conditionals.}
    \item\label{it:gauss_hmm} In $\gauss$, a hidden Markov model is defined by
    \[
        x_t = A_t x_{t-1} + \mathcal{N}\left(v_t, Q_t\right), \qquad 
        y_t = H_t x_t + \mathcal{N}\left(w_t, R_t\right),
    \]
    where $A_t$ and $H_t$ are matrices, and $x_{-1} = 0$ due to $X_{-1} = \monUnit = \R^0$. Thus, $x_t$ and $y_t$ are affine functions of the respective previous state and current state, with Gaussian noise added. These are the classical Gaussian hidden Markov models~\cite{RoweisSam1999AURo}.

  \item\label{it:finsetmulti_hmm} In $\finsetmulti$, a hidden Markov model extends the nondeterministic finite automaton interpretation of Markov chains from \cref{exmp:MarkovChain}. The maps $g_t \colon X_t \to Y_t$ determine the possible outputs for each state. A sequence $y_0, \ldots, y_n$ is possible if and only if there exists a valid sequence $x_0, \ldots, x_n$ such that each $y_t$ is possible according to $g_t$ with input $x_t$.

    This is closely related to the definition of \newterm{labelled transition system}, where one has a single sequence of morphisms
    \[
        h_t \colon X_{t-1} \to X_t \times Y_t
    \]
    which output a state and observation pair at each step.
    Such a labelled transition system is a hidden Markov model in our sense if these morphisms have the specific form
	\begin{equation}
			\label{eq:generalizedHMM}
			\input{GeneralizedHMM.tikz}
	\end{equation}    
	but not in general.
	It would be possible to generalize our definition of hidden Markov models to match the definition of labelled transition system by combining the $f_t$ and $g_t$ into one morphism $h_t$, as some literature does also in the probabilistic setting~\cite{chiganskyCompleteSolutionBlackwell2010}, but we have chosen to stick to the more commonly encountered notion.
\end{enumerate}
\end{exmp}

\section{The Bayes filter in Markov categories}
\label{sec:BayesFilter}

Throughout this section, we work in a Markov category $\catC$ with conditionals and consider a hidden Markov model in the notation of \Cref{defn:HiddenMarkovModel}.
As mentioned, the idea is to model situations in which a system evolves in a Markovian fashion, but the states of the system are only indirectly observable through noisy observations.
A natural problem to ask now is, given a sequence of observations $\left(y_i\right)_{i=0}^t$ up to a time $t$, with $y_i \in Y_i$, what can we infer about the hidden state $x_t \in X_t$?
This is precisely the problem of filtering.

\subsection{The Bayes filter}

Using Bayesian inference, answering the filtering question amounts to conditioning $x_t$ with respect to the given sequence of observations $y_0, \dotsc, y_t$.
Here is the formulation in our categorical setting.

\begin{defn}\label{defn:BayesFilt}
	For $p \colon \monUnit \to \left(Y \otimes X\right)_{\left[n\right]}$ a hidden Markov model, the \newterm{Bayes filter} at time $t \in \left[n\right]$ is the conditional\footnote{See \cref{notn:ConditionalBentWireNotn} for the dashed box and bent wire notation.} 
	\begin{equation}
		\label{eqn:BayesFilterDefn}
		%\tikzfig{jointDist_Xn_prevYs}
		\input{FilterDef.tikz}
	\end{equation}
\end{defn}

The marginalizations on the right-hand side are precisely those over the state spaces $X_0, \dots, X_{t-1}$.
We can therefore also phrase this definition in the following form: the Bayes filter at time $t$ is any morphism $\filtB{t} \colon Y_{\left[t\right]} \to X_t$ such that
\begin{equation}
	\label{eqn:BayesFilterDefnMarginal}
	\input{jointDist_Xn_prevYs.tikz} \quad = \quad \input{conditionalityConstraint_RHS.tikz}
\end{equation}
By \Cref{rmk:ConditionalASUniqueness}, $\filtB{t}$ is often not strictly unique.
For example in $\finstoch$, on a sequence of outcomes $y_0, \ldots, y_t$ which has probability zero, $\filtB{t}$ is completely arbitrary, which intuitively means that we can make an arbitrary inference about the hidden state.
We nevertheless speak of \emph{the} Bayes filter, as the choice is unique up to $\Ch{Y}{t}$-\as{} equality.

Since the definition of $\filtB{t}$ is difficult to work with in practice due to the space $Y_{\left[t\right]}$ that we condition on having size exponential in $t$ (when it is finite), it is useful to have recursion formulas that unravel this complexity.

\begin{prop}\label{prop:B_n}
	The Bayes filter can be recursively computed through conditionals as\footnote{This formulation was suggested to us by Dario Stein.}
	\begin{equation}\label{eq:Bn_rec_defn}
		\input{caseBn.tikz}
	\end{equation}
	for all $t = 0, \dots, n$, where we take $\filtB{-1} \coloneqq \id_{\monUnit}$.
\end{prop}
			
Setting $\filtB{-1} \coloneqq \id_{\monUnit}$ makes sense insofar as before time $0$, there are no observations and no hidden state, so that the Bayes filter at that time takes no input and produces no output, and the first recursion step gives
\begin{equation}
	\label{eq:B0_defn}
	\input{caseB0.tikz}
\end{equation}
This makes $\filtB{0}$ a Bayesian inverse $\left(g_0\right)_{f_0}^\dag$ in the sense of~\cite[Definition 3.5]{chojacobs2019strings}.

\begin{proof}
	Starting with the definition~\eqref{eqn:BayesFilterDefn}, we expand $\Ch{}{t}$ via~\eqref{eqn:Ch_XY_stringDiag} and use \cref{prop:ConditioningCoherence}\ref{cond:CondCoherenceIter} to split it into two conditionals,
	\begin{equation*}
		\input{BayesFiltDefnExpanded.tikz} \quad =_{\as{\Ch{Y}{t}}} \quad \input{BayesFiltDefnExpandedSplit.tikz}
	\end{equation*}
	where we have moved the $Y_t$ input to right to obtain a more convenient way of drawing the diagram.
	An application of \cref{prop:ConditioningCoherence}\ref{cond:CondCoherencePostComp} lets us rewrite the diagram further to
	\begin{equation*}
		=_{\as{\Ch{Y}{t}}}\quad \input{BayesFiltDefnExpandedPostComp.tikz} \quad =_{\as{\Ch{Y}{t}}} \quad \input{caseBnRHS.tikz}
	\end{equation*}
	since the inner conditional is precisely the definition of $\filtB{t-1}$.
	This reasoning is still valid for $t = 0$ if we take all objects and morphisms at time $-1$ to be given by $\monUnit$ and $\id_{\monUnit}$, respectively.
\end{proof}

We will consider instances of the Bayes filter in the usual Markov categories of interest in the subsequent subsections and end here with a few general remarks.

\begin{rmk}\label{rmk:ParamBayesFilt}
	What happens when we instantiate the Bayes filter in a parametric Markov category $\catC_W$, in the style of \Cref{rmk:ParamMChain}?
	This is possible, as our constructions again apply to that case, and we thereby automatically obtain a Bayes filter with parameters.
	%Can it be described more concretely?

	The conditionals in a parametric Markov category $\catC_W$ can be constructed as conditionals in $\catC$ of their representatives \cite[Lemma 2.10]{fritz2023representable}.
	Therefore, the Bayes filter of a hidden Markov model in a parametric Markov category is represented simply by the corresponding conditional in $\catC$. %\footnote{Additionally, \as{}-equalities in $\catC_W$ correspond to \as{}-equalities in $\catC$, in such a way that the uniqueness criteria for the two conditionals coincide.}
	This reproduces the intuitive idea that the Bayes filter for a hidden Markov model with measurable dependence on a parameter should be the Bayes filter for each of the parameter values, and it guarantees that this filter can be chosen to have measurable dependence on the parameter.
\end{rmk}

\begin{rmk}\label{rmk:Virgo}
	A treatment of hidden Markov models and Bayesian filtering in the setting of Markov categories can also be found in a recent paper by Virgo~\cite{virgo2023Unifilar} based on \emph{stochastic Mealy machines}.
    This amounts to considering time-homogeneous hidden Markov models but generalized to the form $X \to Y \otimes X$, where the transition kernel and observation kernel is unified into one morphism as in~\eqref{eq:generalizedHMM}, and in addition also allowing \emph{inputs}\footnote{These can be used to extend the framework to Markov decision processes, for instance.}, which we do not consider in this paper.

	Virgo's work adopts a perspective slightly different from ours.
	Fixing an observation space $Y$, he forms a category $\mathsf{Generator}\left(Y\right)$ having generalized hidden Markov models (in his sense) with observation space $Y$ as objects and morphisms those maps between the state spaces which preserve the structure~\cite[Section 2.1]{virgo2023Unifilar}.
	Assuming that the Markov category in question is \emph{strongly representable}, Virgo then extends Bayesian filtering to a functor out of the category $\mathsf{Generator}\left(Y\right)$.
	This can be seen as functorially converting a hidden Markov model into a machine that updates Bayesian priors.
	It is further shown that this Bayesian filtering functor is a right adjoint~\cite[Theorem 2.6]{virgo2023Unifilar}.
	
	While Virgo's hidden Markov models are partly more restrictive (time-homogeneous) and partly more general than ours, his work is complementary to ours insofar as it is concerned with different aspects of the Bayes filter.
    For instance, Virgo does not prove any version of \cref{prop:B_n}, focusing instead on the adjunction of~\cite[Theorem 2.6]{virgo2023Unifilar} and its consequences (such as constructing conjugate priors).
\end{rmk}

\begin{rmk}\label{rmk:KST}
  In the recent preprint~\cite{tull2023activeinferencestringdiagrams}, Tull, Kleiner and Smithe develop the theory of active inference in the framework of CD categories with \emph{caps} and \emph{normalizations}.
  This is done as a part of an overarching programme to develop a categorical approach to predictive processing.
  The primary focus is to describe an approximate updating procedure based on energy minimization, which is done for a wide class of structures called \emph{open generative models}.
  Of these, the \emph{discrete time open generative models} of Section 3.2 are essentially our hidden Markov models.

	However, the authors focus primarily on the category $\finstoch$ and on a single update step, but entertain greater complexity there by considering updating on non-deterministic observations, which makes Bayesian inversion branch into various versions like Jeffrey updating and Pearl updating.
	Thus their work is complementary to ours, as we focus on the recursive computation of the Bayes filter and smoother in the greater generality of all Markov categories with conditionals.
\end{rmk}

\subsection{The instantiated Bayes filter}
\label{sec:instantiatedBayesFilter}

Per \cref{prop:B_n}, the Bayes filter $\filtB{t}$ is given by a conditional involving $\filtB{t-1}$, $f_t$ and $g_t$.
However, this conditional is still too cumbersome to work with in practice, as its domain $Y_{\left[t\right]} = Y_0 \otimes \cdots \otimes Y_{t-1}$ still typically grows exponentially in $t$. For this reason, we introduce the \newterm{instantiated Bayes filter} here, which specializes the Bayes filter to a \emph{fixed} sequence of observations. This simplifies the recursive computation of the conditional (see \cref{prop:inst_BF}) and recovers standard filters like the Kalman filter in the Gaussian case (see \cref{ssec:KalmanFilter}).
The price to pay is that one does not compute the whole filter $\filtB{t}$, but only its value on a particular sequence of inputs.

\begin{defn}
	Given a sequence of deterministic morphisms
	\[
		y_0 \colon \monUnit \to Y_0, \qquad y_1 \colon \monUnit \to Y_1, \qquad \ldots, \qquad y_t \colon \monUnit \to Y_t,
	\]
	which we abbreviate by $\mathbf{y}_{\left[t\right]}$, the \newterm{instantiated Bayes filter} is
	\[
		\input{instantBayesFilter.tikz}
	\]
\end{defn}

The non-uniqueness of $\filtB{t}$ discussed before now has more severe consequences, as the instantiated Bayes filter is not well-defined in general.
We will discuss this further in the case of $\borelstoch$ in \cref{ssec:instBayesFilterBorelStoch}.

We have the following recursive formula for $\filtIB{t}$ analogous to \Cref{prop:B_n}.

\begin{prop}\label{prop:inst_BF}
	The instantiated Bayes filter can be recursively computed through conditionals as
	\begin{equation}
		\label{eq:instantBayesFilterStep}
		\input{instantBayesFilterStep.tikz}
	\end{equation}
	for all $t = 0, \dots, n$, where we take $\filtIB{-1} \coloneqq \id_{\monUnit}$.
\end{prop}

So explicitly, the $t = 0$ case is given by
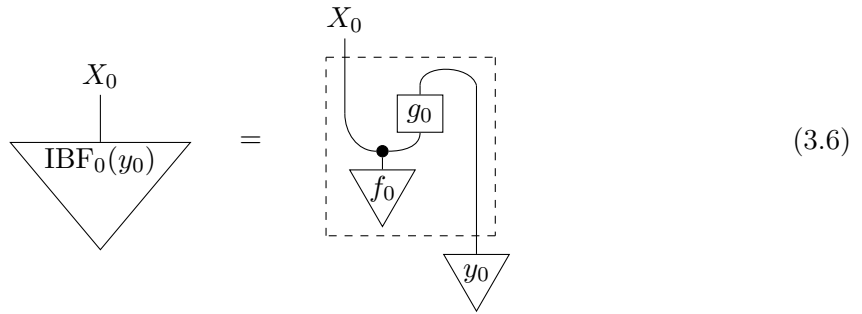
\begin{equation}
	\label{eq:instantBayesFilterStart}
	\input{instantBayesFilterStart.tikz}
\end{equation}
Each step of the recursion involves a \emph{choice} of conditional, and making such choices throughout the recursion merely results in \emph{one version} of the instantiated Bayes filter.

\begin{proof}
	Plugging in \cref{prop:B_n} implies that
	\[
		\input{instantBayesFilter_proof.tikz}
	\]
	where we have used \Cref{prop:ConditioningCoherence}\ref{lem:det_conditional}\footnote{This is where the determinism assumption on the $y_i$ enters.} together with the definition of $\filtIB{t-1}(\mathbf{y}_{\left[t-1\right]})$
\end{proof}

The formula~\eqref{eq:instantBayesFilterStep} allows us to compute our posterior distribution for the hidden state $x_t$ through the following two steps:
\begin{enumerate}
	\item \newterm{Prediction step:} We use the posterior from our previous guess for the hidden state $x_{t-1}$ and apply the transition kernel $f_t$ in order to predict the next hidden state $x_t$.
	\item \newterm{Update step:} We use the observation $y_t$ and our knowledge of the observation kernel $g_t$ to update our guess for the hidden state $x_t$ by conditioning on the observation at time $t$ being $y_t$.
\end{enumerate}
This is the known procedure for the classical Bayes filter~\cite[Theorem~4.1]{Sarkka}, but now generalized to the categorical setting.

\subsection{Examples of the instantiated Bayes filter}
\label{sec:examplesInstantiatedBayesFilter}

In the following, we show that the instantiated Bayes filter recovers known filters for suitable choices of Markov categories. This includes the discrete Bayes filter in the case of $\finstoch$ as well as the Kalman filter in the case of $\gauss$.

\subsubsection{The instantiated Bayes filter in \texorpdfstring{$\finstoch$}{FinStoch}}
\label{ssec:instBayesFilterFinStoch}

We now consider the instantiated Bayes filter in the Markov category $\finstoch$ explicitly, a procedure sometimes called the \newterm{forward algorithm}\footnote{See \Cref{rmk:inconsistencies}.}.
As in \cref{exmp:HMM}, the hidden Markov model up to time $t \in \left[n\right]$ is a probability distribution on the product of the finite sets $\left(X_i\right)_{i=0}^t$ and $\left(Y_i\right)_{i=0}^t$.
It is ``categorical'' in the sense of these sets being finite.
The sequence of deterministic morphisms $\mathbf{y}_{\left[t\right]}$ is a sequence of observed values $y_0 \in Y_0, \ldots, y_t \in Y_t$.
To then compute the instantiated Bayes filter $\filtIB{t}$ recursively, we consider the prediction and update steps as explained above separately.
The prediction step produces the distribution $q_t \colon \monUnit \to X_t$ given by
\begin{equation}
\label{eq:filter_step_finstoch}
	q_t\left(x_t\right) \coloneqq \sum_{x_{t-1}} f_t\left(x_t \,| \,x_{t-1}\right) \cdot \filtIB{t-1}\left(\mathbf{y}_{\left[t-1\right]}\right)\left(x_{t-1}\right)
\end{equation}
for $t > 0$, and $q_0 = f_0$ to get the recursion going.
Performing the update step on top of that results in
\begin{equation}
	\label{eq:xn_estimate}
	\filtIB{t}\left(\mathbf{y}_{\left[t\right]}\right)\left(x_t\right) = \frac{g_t\left(y_t \,| \,x_t\right) \,q_t\left(x_t\right)}{\sum_{x'_t} g_t\left(y_t \,| \,x'_t\right) \,q_t\left(x'_t\right)}.
\end{equation}
This equation agrees with the Bayes filter in its standard form for discrete variables~\cite[Eq.~(1.25)]{ristic2004Beyond}.
It is worth noting that the right-hand side is well-defined whenever the denominator is nonzero, while $\filtIB{t}\left(\mathbf{y}_{\left[t\right]}\right)$ is an arbitrary distribution whenever the denominator is zero.
This is not an issue in the present discrete case since the denominator is nonzero for all $y_t$ that can occur as part of a sequence of observations with nonzero probability.

\subsubsection{The instantiated Bayes filter in \texorpdfstring{$\borelstoch$}{BorelStoch}}
\label{ssec:instBayesFilterBorelStoch}

This construction immediately generalizes to $\borelstoch$.
So now our $X_i$ and $Y_i$ are standard Borel spaces, and the $f_i\colon X_{i-1} \to X_i$ and $g_i\colon X_i \to Y_i$ are measurable Markov kernels.
Again we have a sequence of observations $y_0 \in Y_0, \dots, y_{t-1} \in Y_{t-1}$.
For $t > 0$, the prediction step constructs a probability measure $q_t \colon \monUnit \to X_t$ given by
\[
	q_t\left(S \right) \coloneqq \int_{x_{t-1} \in X_{t-1}} f_t\left( S \,| \,x_{t-1}\right) \cdot \filtIB{t-1}\left(\mathbf{y}_{\left[t-1\right]}\right)\left(\d x_{t-1}\right),
\]
for every measurable $S \subseteq X_t$, which is the generalization of~\eqref{eq:filter_step_finstoch} in $\borelstoch$, and with initial condition $q_0 \coloneqq f_0$.
The update step is more subtle: every individual observation value $y_t$, or entire observation sequence $\mathbf{y}_{\left[t\right]}$, typically has probability zero, and we run into the issue that conditioning on events of probability zero is ill-defined.\footnote{Compare \Cref{rmk:ConditionalASUniqueness} and the Borel--Kolmogorov paradox~\cite[Section~15.7]{jaynes2003probability}, which illustrates that this lack of uniqueness is not an issue with our formalism but rather an unavoidable complication in measure-theoretic probability quite generally.}
Hence, the instantiated Bayes filter $\filtIB{t}\left(\mathbf{y}_{\left[t\right]}\right)$ is completely arbitrary for every probability zero sequence $\mathbf{y}_{\left[t\right]}$, which is the typical case with continuous variables.
What can be considered more well-defined---as per the uniqueness of conditionals up to almost sure equality---is the \emph{totality} of values $\filtIB{t}\left(\mathbf{y}_{\left[t\right]}\right)$ for all $\mathbf{y}_{\left[t\right]}$, which is our Bayes filter $\filtB{t}$ again.

With this in mind, we can explicate the update step: it is given by the formation of product regular conditional probabilities, which amounts to the implicit equation
\[
	\int_{y_t \in T} \int_{x_t \in X_t} \filtIB{t}\left(\mathbf{y}_{\left[t\right]}\right)\left(S\right) \,g_t\left(\d y_t \,| \,x_t\right) \,q\left(\d x_t\right) = \int_{x_t \in S} g_t\left(T \,| \,x_t\right) \,q\left(\d x_t\right)
\]
for all measurable $S \subseteq X_t$ and $T \subseteq Y_t$.
This follows as an instance of~\cite[Example~11.3]{fritz2019synthetic}, and can be interpreted as follows: both sides represent the probability that the hidden state $x_t$ is in $S$ and that the observation $y_t$ is in $T$, and $\filtIB{t}\left(\mathbf{y}_{\left[t\right]}\right)$ is constructed precisely such that this probability can be calculated as on the left-hand side.

\cite[Proposition 3.1.4]{cappe2005inference} gives a similar filtering formula for hidden Markov models of continuous variables with the additional restriction that the processes $g_t\left(\d y_t \,| \,x_t\right)$ need to be absolutely continuous for all $x_t$ with respect to a single measure. 

\subsubsection{The instantiated Bayes filter in \texorpdfstring{$\finsetmulti$}{FinSetMulti}}
\label{ssec:instBayesFilterSetMulti}

We turn to the \newterm{possibilistic Bayes filter}, by which we mean the Bayes filter in $\finsetmulti$.
This seems to have not been considered before. 

Given a hidden Markov model as in \cref{exmp:HMM}\ref{it:finsetmulti_hmm}, the prediction step for the instantiated Bayes filter again takes the form
\begin{equation}
	\label{eq:filter_pred_finsetmulti}
	q_t\left(x_t\right) \coloneqq \sum_{x_{t-1}} f_t\left(x_t \,| \,x_{t-1}\right) \cdot \filtIB{t-1}\left(\mathbf{y}_{\left[t-1\right]}\right)\left(x_{t-1}\right).
\end{equation}
This now means that $x_t$ is considered possible if and only if $x_{t-1}$ is possible according to our previous estimate of the hidden state $\filtIB{t-1}\left(\mathbf{y}_{\left[t-1\right]}\right)$ and $x_t$ is possible according to the transition kernel $f_t$.
Choosing conditionals as in \Cref{exmps:CatsWithConditionals}\ref{cond:FinSetMulti} and using the general prescription shows that the update step is given simply by
\begin{equation}
	\label{eq:filter_update_finsetmulti}
	\filtIB{t}\left(\mathbf{y}_{\left[t\right]}\right)\left(x_t\right) = g_t\left(y_t \,| \,x_t\right) \cdot q_t\left(x_t\right).
\end{equation}
Since this multiplication amounts to a logical conjunction of possibilities, this now means that the state $x_t$ is considered possible if and only if it was considered possible before and is in addition compatible with the observed $y_t$ according to the observation map $g_t$.

\begin{exmp}
	\label{exmp:filter_finsetmulti}
Let us work through a simple example of a hidden Markov model in $\finsetmulti$ and the possibilistic filter.
Consider first the time-homogeneous Markov chain---or equivalently as a nondeterministic automaton---on a three-element set $X = \{a, b, c\}$ with possible transitions like this:
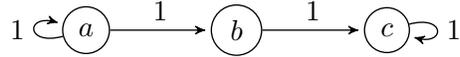
\begin{center}
\begin{tikzpicture}[->,>= stealth',shorten >=2pt, line width =0.5 pt , node distance =2cm]
\node [circle,draw] (zero) {$a$};
\node [ circle , draw ] (one) [ right of = zero ] {$b$};
\node [ circle , draw ] (two) [ right of = one ] {$c$};
\path (zero) edge [ loop left ] node {$1$} (zero) ;
\path (zero) edge node [ above ] {$1$} (one) ;
\path (one) edge node [ above ]{$1$} (two) ;
\path (two) edge [ loop right ] node {$1$} (two) ;
\end{tikzpicture}
\end{center}
Our transition kernels $f_t$ are all the same for $t > 0$ by time-homogeneity, and we thus omit the index $t$.
In matrix form, the transition kernel $f$ is given by\footnote{So $f(x'|x) = 1$ if it is possible to reach $x'$ from $x$ in one step. This matrix representation is the transpose of transition matrices as conventionally used for Markov chains in discrete probability theory.}
\[
	f = (f(x'|x))_{x,x' \in \{a,b,c\}} = \begin{pmatrix}
	1 & 0 & 0 \\
	1 & 0 & 0 \\
	0 & 1 & 1
\end{pmatrix}.
\]
Suppose that our chain starts at $a$, which means that our initial state is 
$f_0 = [	1, 0, 0 ]^t $.

Let us consider observation morphisms $g_t \colon \{a,b,c\} \to \{\bot,\top\}$ that are also independent of $t$ and given by the matrix
\[
	g = (g(y|x))_{y \in \{\bot,\top\}, x \in \{a,b,c\}} = \begin{pmatrix}
	1 & 1 & 0 \\
	0 & 1 & 1
\end{pmatrix}. 
\]
Thus the hidden state $a$ is always observed as $\bot$, the hidden state $c$ as $\top$, and $b$ as either.
Given the underlying Markov chain, it is clear that the possible sequences of observations are as follows: a sequence of $\bot$'s as long as the hidden state is $a$, possibly followed by an arbitrary sequence of $\bot$ and $\top$ while it is $b$, possibly followed by only $\top$'s if the chain proceeds to its terminal state $c$.

Now, let's compute the instantiated Bayes filter by doing the prediction and update steps.
For time $t = 0$, the prediction coincides with the initial distribution: $q_0 = f_0$.
On $y_0 = \bot$ as the only possible initial observation, the initial update step gives
\[
\filtIB{0}(\bot) =  
\begin{pmatrix}
	g(\bot|a) \, q_0(a) \\
	g(\bot|b) \, q_0(b) \\
	g(\bot|c) \, q_0(c)
\end{pmatrix} = 
\begin{pmatrix}
	1 \\
	0 \\
	0 
\end{pmatrix}.
\]
This tells us what we already knew: the initial state must have been $a$ with certainty.
The other instance $\filtIB{0}(\top)$ depends on the choice of conditional, since $y_0 = \top$ is not possible under the initial state $a$.
Using the choice of conditional~\eqref{eq:filter_update_finsetmulti}, we get $\filtIB{0}(\top) = \filtIB{0}(\bot)$, in line with the fact that the initial observation does not tell us anything new about the hidden state.

Suppose now that our sequence of observations up to $t = 2$ is
$
	\mathbf{y}_{[2]} = (\bot, \bot, \top).
$
At time $t = 1$, the chain can a priori be in state $a$ or $b$.
Indeed using the $\finsetmulti$-arithmetic $1 + 1 = 1$ in the prediction step~\eqref{eq:filter_pred_finsetmulti} gives
\[
q_1 = f \;\; \filtIB{0}(\bot) =  \begin{pmatrix}
	1 & 0 & 0 \\
	1 & 0 & 0 \\
	0 & 1 & 1
\end{pmatrix} 
\begin{pmatrix}
	1 \\
	0 \\
	0 
\end{pmatrix} = \begin{pmatrix}
	1 \\
	1 \\
	0 
\end{pmatrix}.
\]
Applying the update step~\eqref{eq:filter_update_finsetmulti} produces
\[
\filtIB{1}(\bot, \bot) =  
\begin{pmatrix}
	g(\bot|a) \, q_1(a) \\
	g(\bot|b) \, q_1(b) \\
	g(\bot|c) \, q_1(c)
\end{pmatrix} = 
\begin{pmatrix}
	1 \\
	1 \\
	0 
\end{pmatrix}.
\]
So with our second observation having been $y_1 = \bot$, we cannot distinguish between the hidden states $a$ and $b$, but we can be sure that the chain is not in state $c$.

Now at time $t=2$, and before we take the observation $y_2 = \top$ into account, the chain can be in any of the three states: the prediction step~\eqref{eq:filter_pred_finsetmulti} gives
\[
q_2 = f \;\; \filtIB{1}(\bot, \bot) =  \begin{pmatrix}
	1 & 0 & 0 \\
	1 & 0 & 0 \\
	0 & 1 & 1
\end{pmatrix} 
\begin{pmatrix}
	1 \\
	1 \\
	0 
\end{pmatrix} = \begin{pmatrix}
	1 \\
	1 \\
	1 
\end{pmatrix}.
\]
At this point, the third observation $y_2 = \top$ finally gives us nontrivial information:  using the update step~\eqref{eq:filter_update_finsetmulti}, we find
$
\filtIB{2}(\bot, \bot, \top) = (0,1,1)^t,
%\begin{pmatrix}
%	0 \\
%	1 \\
%	1 
%\end{pmatrix},
$
which again is intuitively clear as the observed value $y_2 = \top$ has excluded the hidden state $a$.
\end{exmp}

\subsubsection{The Kalman filter as the instantiated Bayes filter}\label{ssec:KalmanFilter}

The Kalman filter is the special case of the Bayes filter for Gaussian distributions.
It was initially formulated as an optimization problem~\cite{kalman1960filtering, kalman1961filtering} and considered from a Bayesian perspective only later~\cite{ho1964bayesian}.
This re-derivation utilized a version of Bayes' rule for probability density functions: the equations for the prediction and update steps were obtained from the parametrization of the probability densities for multivariate Gaussian distributions.

Here, we show that in the Markov category $\gauss$, our instantiated Bayes filter turns into a mildly generalized version of the Kalman filter.
So as in~\cref{exmp:HMM}\ref{it:gauss_hmm}, the ${\left(x_t\right)}_{t=0}^n$ and ${\left(y_t\right)}_{t=0}^n$ are Gaussian vectors which are related via
\[
	x_{t} = A_t x_{t-1} + \mathcal{N}\left(v_t, Q_t\right), \qquad y_{t} = H_t x_{t} + \mathcal{N}\left(w_t, R_t\right),
\] 
where our notation loosely follows that of~\cite[Section~4.3]{Sarkka}.
%The instantiated Bayes filter at step $n-1$ is given by $X_{n-1} \sim \mathcal{N}(\mathbf{x}_{n-1}, P_{n-1})$.
In particular, we write $m_t$ and $P_t$ for the mean and covariance matrix of the Gaussian distribution of $\filtIB{t}\left(\mathbf{y}_{\left[t\right]}\right)$, so that
\[
	\filtIB{t}\left(\mathbf{y}_{\left[t\right]}\right) = \mathcal{N}\left(m_t, P_t\right).
\]
We may think of $m_t$ as our best guess of the hidden state given the observations up to time $t$, while the covariance matrix $P_t$ measures our uncertainty in this guess.
We suppress the dependence on $\mathbf{y}_{\left[t\right]}$ from the notation.
The problem solved by the Kalman filter is the calculation of $m_t$ and $P_t$ recursively from $m_{t-1}$ and $P_{t-1}$ together with the new observation $y_t$.

For the prediction step, we simply apply the transition kernel $f_t$ to $\filtIB{t-1}\left(\mathbf{y}_{\left[t-1\right]}\right)$.
By the composition formula in $\gauss$~\eqref{eq:gauss_composition}, this produces a new Gaussian with parameters
\begin{equation}
	\label{eq:gauss_prediction}
	\tilde{m}_t = A_t m_{t-1} + v_t, \qquad \tilde{P}_t = A_t P_{t-1} A_t^t + Q_t,
\end{equation}
where the tildes remind us that this models the prediction obtained from the filter at time $t-1$ before taking the observation $y_t$ into account.
The joint distribution with the expected observation $y_t$, or more formally the right-hand side of~\eqref{eq:instantBayesFilterStep} before the conditioning, can then be calculated to be given by the Gaussian
\[
	\begin{pmatrix}
		x_t \\ y_t
	\end{pmatrix} = \mathcal{N}\left(a_t, C_t\right)
\]
with parameters
\[
	a_t = \begin{pmatrix}
		\tilde{m}_t \\ H_t \tilde{m}_t + w_t
		\end{pmatrix} \qquad \text{ and }
		\qquad
	C_t = \begin{pmatrix}
		\tilde{P}_t & \tilde{P}_t H_t^t \\
		H_t \tilde{P}_t & S_t
		\end{pmatrix},
\]
where
\begin{equation}
	\label{eq:innovationCovariance}
	S_t \coloneqq H_t \tilde{P}_t H_t^t + R_t
\end{equation}
is the so-called \newterm{innovation covariance}.
Now $\filtIB{t}\left(\mathbf{y}_{\left[t\right]}\right)$ is given by the conditional of this Gaussian with respect to $y_t$.
Using the known formulas for conditionals in $\gauss$ from \cref{exmps:CatsWithConditionals}\ref{cond:Gauss} and introducing the \newterm{optimal Kalman gain}\footnote{As in \cref{exmps:CatsWithConditionals}\ref{cond:Gauss}, we write $S_t^{-}$ for the Moore-Penrose pseudoinverse of $S$. In practice $S_t$ will usually be invertible, and then simply $S_t^- = S_t^{-1}$.}
\begin{equation}
	\label{eq:kalmanGain}
	K_t = \tilde{P}_t H_t^t S_{t}^{-}
\end{equation}
as shorthand notation, we obtain that
\[
	\filtIB{t}\left(\mathbf{y}_{\left[t\right]}\right) = \mathcal{N}\left(m_t, P_t\right)
\]
with
\[
	m_t = \tilde{m}_t + K_t \left(y_t - H_t \tilde{m}_t - w_t\right), \qquad 
	P_t = \left(I - K_t H_t\right) \tilde{P}_t.
\]
Together with~\eqref{eq:gauss_prediction},~\eqref{eq:innovationCovariance}, and~\eqref{eq:kalmanGain}, these formulas form a mild generalization of the standard Kalman filter formulas~\cite[Theorem~4.2]{Sarkka}: the latter are recovered as the special case $v_t = w_t = 0$, which amounts to assuming zero displacement and centered Gaussian noise for both the transition kernels and the observation kernels.

There are many other variations on the Kalman filter. 
For example, some Kalman filter variations use different representations of uncertainty to improve robustness to numerical error from matrix inversions in the Bayesian inversion step~\cite{bierman1977factorization}.
Others account for different structures in the state and measurement spaces, such as filters over Riemannian manifolds~\cite{menegaz2019unscented}.
Many variations attempt to handle nonlinear transition kernels and observation maps by approximating their state propagations~\cite{bar2004estimation,julier1997new}.
Since the general predict-update structure is common to all of these, it is plausible that at least some of them can be captured by our current framework by choosing the appropriate Markov category, while others may require further generalizations.

\subsection{The Bayes filter recursion as a single string diagram}
\label{ssec:BayesFilterRecursion}

For a given a sequence of observations as input, our Bayes filter $\filtB{t}$ should be thought of as returning a \emph{sample} from the posterior distribution over the hidden state.
Although that distribution itself is precisely what is specified by $\filtB{t}$, it is not directly accessible as an output if one thinks of $\filtB{t}$ as a process that produces a random output.
This situation changes in a representable Markov category (\cref{defn:Representability}).
There $\filtB{t}$ has a deterministic counterpart $\filtB{t}^{\sharp} \colon Y_{\left[t\right]} \to \Dist X_t$, which can be thought of as the actual function that returns the posterior distribution and which recovers $\filtB{t}$ after composition with $\samp{} \colon \Dist X_t \to X_t$.

Throughout this section, we assume that our Markov category is $\as{}$-compatibly representable (\cref{defn:ASCompRep}) in addition to having conditionals.
By \cref{rmk:ConditionalASUniqueness}, this ensures that the deterministic counterpart $\filtB{t}^{\sharp}$ is well-defined up to $\as{\Ch{Y}{t}}$ equality.
Consider a hidden Markov model in the form of \Cref{prop:MarkovProperties}\ref{cond:rep} as before, our goal is then to rewrite the recursion for the Bayes filter in terms of $\filtB{t}^{\sharp}$.
This will allow us to draw it as a single string diagram.

\begin{notn}\label{notn:TransNUpdate}
	For $t = 0, \dots, n$, let $\update{t}\colon \Dist X_{t-1} \otimes Y_t \to X_t$ be the morphism
	\begin{equation}
		\label{eq:un_def}
		\input{un_def.tikz}
	\end{equation}
\end{notn}

As before, this uses the convention that $X_{-1} = \monUnit$, and thus also $\Dist X_{-1} = \monUnit$, so that $\update{0} = \filtB{0}$.
Here is how the $u_t$ act as update maps for the Bayes filter.

\begin{prop}\label{prop:BayesUpdate}
	We have
	\begin{equation}\label{eqn:BayesUpdateNonDet}
		\input{BnBoxDiag.tikz} \qquad =_{\as{\Ch{Y}{t}}}\qquad \input{bayesAltCon.tikz}
	\end{equation}
	for all $t = 0, \dots, n$.
\end{prop}

\begin{proof}
	This is a consequence of the equations
	\begin{align*}
		\input{iteration_BFsharp_proof2.tikz} \\
		\input{iteration_BFsharp_proof3.tikz} \qquad\quad
	\end{align*}
	where we have used \cref{prop:B_n} in the first step and \Cref{prop:ConditioningCoherence}\ref{lem:det_conditional} together with the fact that $\filtB{n-1}^{\sharp}$ is deterministic in the third step.
\end{proof}

\begin{cor}\label{cor:BayesUpdateDet}
	The deterministic counterpart of the Bayes filter satisfies the recursion
	\begin{equation}
		\label{eqn:BayesUpdateDet}
		\input{BnBoxDiagDet.tikz} \qquad =_{\as{\Ch{Y}{t}}} \qquad \input{bayesAltConDet.tikz}
	\end{equation}
\end{cor}

\begin{proof}
	Both sides of the equation are deterministic, and post-composition with $\samp{}$ shows that they are the deterministic counterparts of the two sides of~\eqref{eqn:BayesUpdateNonDet}.
	Therefore the claim follows by $\as{}$-compatible representability (\cref{defn:ASCompRep}).
\end{proof}

%\begin{rmk}\label{rmk:NonASCompUpdating}
%	As the proof of \cref{cor:BayesUpdateDet} shows, we may reduce the \as{}-compatible representability assumption to mere representability if we restrict ourself to that particular choice of $\filtB{t}$ which makes~\eqref{eqn:BayesUpdateNonDet} hold strictly.
%
%	Thus, even with just general representability, \cref{prop:BayesUpdate} and \cref{cor:BayesUpdateDet} can be seen as in \cref{rmk:ConstFromASE} as describing constructions of a particular choice of $\filtB{t}$ (and hence $\filtB{t}^\sharp$) in terms of pre-existing choices of $\filtB{t-1}$ and $u_t$.
%\end{rmk}

Since the recursion of \cref{cor:BayesUpdateDet} is easy to expand, we immediately arrive at the following.

\begin{cor}
	\label{cor:BayesFilterSingleDiagram}
	In any \as{}-compatibly representable Markov category with conditionals, the Bayes filter can be written as the single string diagram
	\[
		\input{BnBoxDiagDet.tikz} \qquad =_{\as{\Ch{Y}{t}}} \qquad \input{BFasU.tikz}
	\]
	which furthermore recovers $\filtB{t}$ itself upon composition with $\samp{}$.
\end{cor}

\begin{rmk}\label{rmk:ParamHMMShrp}
	If $\catC$ is $\as{}$-compatibly representable, then so is the parametric Markov category $\catC_W$ for any object $W$~\cite[Lemma~3.25]{fritz2023representable}.
	Further, the deterministic counterpart of a morphism in $\catC_W$ is represented in $\catC$ by the deterministic counterpart of its representative in $\catC$.
	Consequently, as with \cref{rmk:ParamBayesFilt}, the results of this section also apply to $\catC_W$.
	This immediately shows that \cref{cor:BayesFilterSingleDiagram} holds even in case all morphisms involved have an additional dependence on some parameter object $W$.
	This amounts to every box in our diagrams, including the $\update{i}$, connecting to a copy of the $W$ input as in~\eqref{eqn:Ch_Xn_stringDiagParam}.
\end{rmk}

\subsection{The Bayes filter as a Markov chain}
\label{sec:bayes_markov}

Consider a hidden Markov model and the sequence of posterior distributions that describe an observer's guess about the hidden state.
Since the observations are typically random, the posterior distributions are also random.
It is natural to consider them as a stochastic process in their own right.
Here we show that this process is a Markov chain. 
As far as we know, this is due to Blackwell~\cite{blackwell1957markov} for the case of $\finstoch$, while~\cite[Lemma~2.4]{chiganskyCompleteSolutionBlackwell2010} provides a partial measure-theoretic generalization.

Throughout, we work in an $\as{}$-compatibly representable Markov category $\catC$ with conditionals and consider a hidden Markov model and its Bayes filter as before.
In these terms, the stochastic process of posterior distributions up to some time $t \in [n]$ is given schematically by the \newterm{filter process}
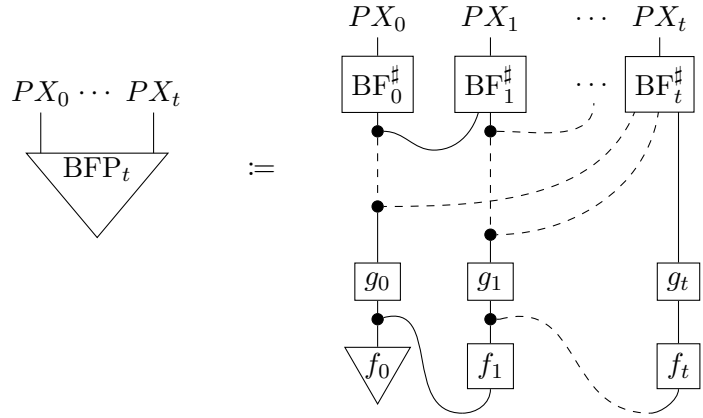
\[
	\input{ChBStateDiag.tikz} \qquad \coloneqq \qquad \input{BayesJointn.tikz}
\]
where the dashed wires indicate what the pattern is.
In order to spell this out more formally, we introduce the following auxiliary morphisms.

\begin{defn}\label{defn:lambda}
	Let
	\[
		\lambda_t\colon Y_0\otimes \dots \otimes Y_t \longrightarrow \Dist X_0 \otimes \dots \otimes \Dist X_t
	\]
	be defined recursively, starting with $\lambda_0 \coloneqq \filtB{0}^{\sharp}$, and for $t \ge 1$,
	\[
		\input{lambdanBox.tikz} \qquad \coloneqq \qquad \input{lambdanCon.tikz}
	\]
\end{defn}

As before, we could alternatively consider the start of the recursion to be $\lambda_{-1} \coloneqq \id_{\monUnit}$ with $Y_{-1} = I$ and $\Dist X_{-1} = X_{-1} = \monUnit$.
In terms of these $\lambda_t$'s, the filter process is now formally defined as
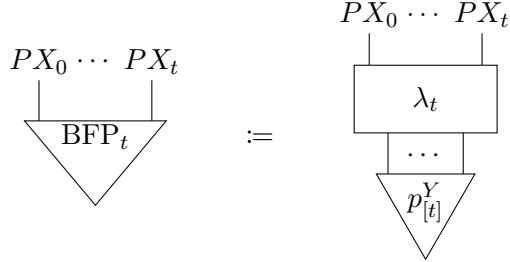
\[
	\input{ChBStateDiag.tikz} \qquad \coloneqq \qquad \input{ChBExp.tikz} 
\]

\begin{lem}\label{lem:lambdaBF}
	For all $t \in \left[n\right]$, we have
	\[
		\input{lambdaBF.tikz} \qquad = \qquad \input{lambdaCopy.tikz}
	\]
\end{lem}

\begin{proof}
	For $t = 0$, this is simply the statement that $\lambda_0 = \filtB{0}^{\sharp}$ is deterministic.
	The induction step follows similarly straightforwardly by the fact that $\filtB{t}^{\sharp}$ is deterministic.\footnote{For a more abstract proof, note that all morphisms involved are deterministic, and it is therefore enough to show that all single-output marginals of the equation hold. We leave the details to the reader.}
\end{proof}

The following other sequence of auxiliary morphisms will describe the transition kernels of the filter process.

\begin{defn}\label{defn:filtTrans}
	For $t \in \left[n\right]$, let $\filtTrans{t}$ be the morphism
	\[
		\input{hnBoxDiag.tikz} \qquad \coloneqq \qquad \input{filtTransFig.tikz}
	\]
	For $t=0$, we follow the convention $X_{-1}\coloneqq \monUnit$, so that $\filtTrans{0} = \update{0}^{\sharp} g_0 f_0$.
\end{defn}

\begin{rmk}
	Intuitively, $\filtTrans{t}$ takes a distribution on $X_{t-1}$, samples from it and determines a sample of the observation in $Y_t$, and then applies a Bayesian updating to the original distribution with respect to this observation.
	In $\finstoch$, our $h_t^\sharp$ seems to match up with the ``abstract hidden Markov models'' of McIver et al.~described at~\cite[Section~3.B]{mciver2019monadic} and~\cite[Section~14.3]{alvim2020quantitative}.
\end{rmk}

Before proceeding, it will be helpful to note that the process of observations can be generated recursively using the Bayes filter in the following sense.

\begin{lem}\label{lem:ObsJoint}
	For all $t \in \left[n\right]$, we have
	\[
		\input{ObsJoint.tikz} \qquad = \qquad \input{ObsJointTrans.tikz}
	\]
\end{lem}

\begin{proof}
	At $t = 0$, this amounts to $p^Y_{\left[0\right]} = g_0 f_0$, which is obvious.
	And with $t + 1$ in place of $t$, the claim follows upon composing the definition of the Bayes filter~\eqref{eqn:BayesFilterDefn} with $g_{t+1} f_{t+1}$.
\end{proof}

The following result now shows that the filter process is indeed a Markov chain with initial distribution $\filtTrans{0}$ and transition kernels $\filtTrans{t}$.

\begin{thm}\label{prop:ChBMarkov} For all $t\in \left[n\right]$,
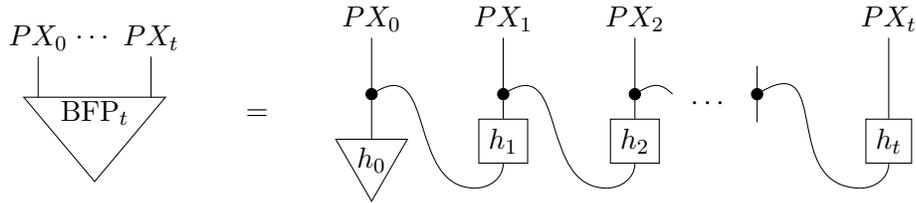
	\[
		\input{ChBStateDiag.tikz} \qquad = \qquad \input{Ch_Bn_stringDiag.tikz}
	\]
\end{thm}
\begin{proof}
	We proceed by induction on $t$.
	The base case $t = 0$ asserts $\filtB{0}^{\sharp} p^Y_0 = \filtTrans{0}$, which holds because of $\update{0} = \filtB{0}$ and $p^Y_0 = g_0 f_0$.

	For the induction step, we consider $t \ge 1$ such that the desired equation holds for $\filtBP{t-1}$, and we then show that it holds for $\filtBP{t}$.
	Expanding the definitions gives
	\[
		\input{ChBStateDiag.tikz} \qquad = \qquad \input{ChBExp.tikz} \qquad = \qquad \input{ChBAltExp.tikz}
	\]
	Applying \Cref{lem:ObsJoint}, this evaluates further to
	\[
		\qquad = \qquad \input{ChBUpdateElimd.tikz} \qquad = \qquad \input{ChBSplitLambdaPacked.tikz}
	\]
	where the second equation uses the recursion formula of \Cref{cor:BayesUpdateDet}.
	By the fact that $\filtB{t-1}^{\sharp}$ is the deterministic counterpart of $\filtB{t}$, we can also write this as
	\[
		\quad = \quad \input{ChBSplitBSharpPushed.tikz} \quad = \quad \input{ChBSplitLambdaTrans.tikz} \quad = \quad \input{ChBTrans.tikz}
	\]
	where the second equation holds by \Cref{lem:lambdaBF} and the definition of $\filtTrans{t}$.
	The claim now follows straightforwardly by the induction hypothesis.
\end{proof}

\begin{rmk}
	This result may have some significance in the context of information security.
	Indeed, if the observations in the hidden Markov model are the observations of an adversary, then it is of interest to make statements about the adversary's guess of the hidden state and how it evolves in time.
	The fact that the filter process is a Markov chain indicates that given the adversary's guess at some time $t$, learning about their guess at an earlier time does not help with predicting their guess at a later time.
\end{rmk}

\begin{exmp}
	In $\finsetmulti$, each distribution object $P X_t$ is the set of nonempty subsets of $X_t$.
	Indeed evaluating~\eqref{eq:un_def} in $\finsetmulti$ shows that the filter update
	\[
		\update{t}^\sharp \: \colon \: PX_{t-1} \otimes Y_t \longrightarrow PX_t
	\]
	is the map which takes a subset of $X_{t-1}$ and an observation $y_t$, applies the transition $f_t$ to determine all states $x_t$ consistent with the given set at time $t-1$ and then intersects with the set of states consistent with the observed value $y_t$.
	This is effectively what we already did in \Cref{exmp:filter_finsetmulti} to compute the possibilistic filter.

	Continuing on from that example, let us see what the filter process and its Markov property amount to in that example.
	At $t = 0$, we have seen that the filter deterministically outputs
	$
  \filtIB{0}(\bot) = (1,0,1)^t. 
		%\begin{pmatrix}
		%	1 \\
		%	0 \\
		%	0 
		%\end{pmatrix}.
	$
	At time $t = 1$, the possible filter outputs are
	 $
		\filtIB{1}(\bot, \bot) = (1,1,0)^t
    % \begin{pmatrix}
		% 	1 \\
		% 	1 \\
		% 	0 
		% \end{pmatrix}
    $ and $
		\filtIB{1}(\bot, \top) = (0,1,0)^t
    % \begin{pmatrix}
		% 	0 \\
		% 	1 \\
		% 	0
		% \end{pmatrix},
	  $
	depending on the value of the second observation.
	At time $t = 2$, we similarly get one of
	$
		\filtIB{2}(\bot, \bot, \bot) = (1,1,0)^t,
    % \begin{pmatrix}
		% 	1 \\
		% 	1 \\
		% 	0 
		% \end{pmatrix}, \qquad 
		\filtIB{2}(\bot, \bot, \top) = (0,1,1)^t
    % \begin{pmatrix}
		% 	0 \\
		% 	1 \\
		% 	1
		% \end{pmatrix}, \qquad
    $, and $
		\filtIB{2}(\bot, \top, \top) = (0,0,1)^t.
    % \begin{pmatrix}
		% 	0 \\
		% 	0 \\
		% 	1
		% \end{pmatrix}.
	$
	Continuing like this, a simple induction argument shows that the filter process is given by the following possibilistic Markov chain, or nondeterministic automaton:
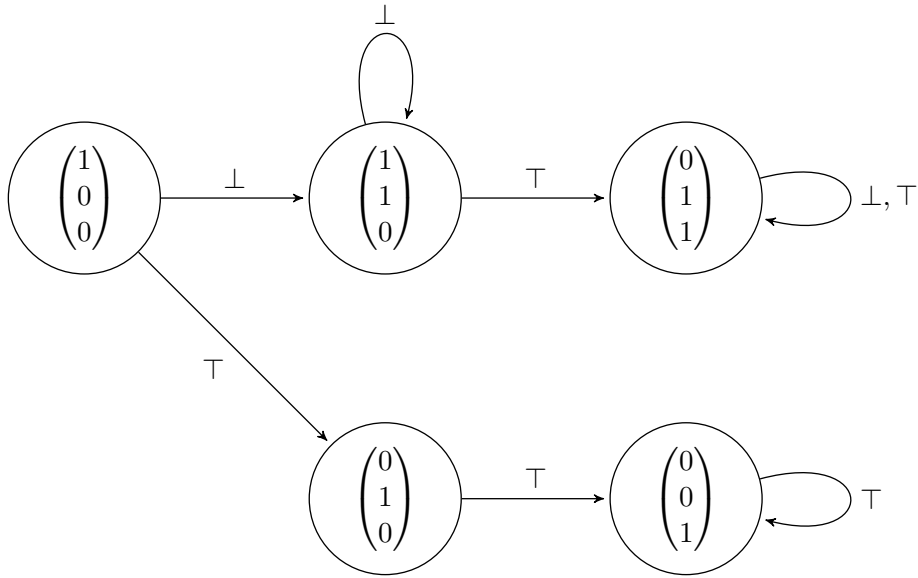
\begin{center}
\begin{tikzpicture}[->,>= stealth', shorten >=2pt, line width=0.5 pt, node distance=0.5cm and 1cm, oval/.style={ellipse,draw}]
\node [oval,draw] (zero) {$(1,0,0)^t$};
%{$\begin{pmatrix} 1 \\ 0 \\ 0 \end{pmatrix}$};
\node [oval , draw ] (one) [ right = of zero ] {$(1,1,0)^t$};
%{$\begin{pmatrix} 1 \\ 1 \\ 0 \end{pmatrix}$};
\node [oval , draw ] (two) [ right = of one ] {$(0,1,1)^t$};
%{$\begin{pmatrix} 0 \\ 1 \\ 1 \end{pmatrix}$};
\node [oval , draw ] (three) [ below = of one ] {$(0,1,0)^t$};
%{$\begin{pmatrix} 0 \\ 1 \\ 0 \end{pmatrix}$};
\node [oval , draw ] (four) [ below = of two ] {$(0,0,1)^t$};
%{$\begin{pmatrix} 0 \\ 0 \\ 1 \end{pmatrix}$};
\path (zero) edge node [ above ] {$\bot$} (one) ;
\path (one) edge node [ above ] {$\top$} (two) ;
\path (one) edge [ loop above ] node {$\bot$} (one) ;
\path (two) edge [ loop right ] node {$\bot, \top$} (two) ;
\path (zero) edge node [ below left ] {$\top$} (three) ;
\path (three) edge node [ above ] {$\top$} (four) ;
\path (four) edge [ loop right ] node {$\top$} (four) ;
\end{tikzpicture}
\end{center}
\end{exmp}
Here, the labels indicate which new observations trigger which transition.
$\bot$ does not appear on any transition out of the bottom two states, since when we are at $b$ or $c$ with certainty, then we must be at $c$ in the next step and will necessarily observe $\top$.

\section{Smoothing in Markov categories}
\label{sec:smoothing}

While the Bayes filter tries to infers the most recent state on $X_t$ from a sequence of observations up to time $t$, it is sometimes desirable to go back and make inferences about the hidden state at an earlier time.
So in this section, we consider observations up to time $n$ and the problem of inferring the hidden state at a time $t \in \left[n\right]$.

Some of the material in this section will be similar to that in \cref{sec:BayesFilter}, so we will be brief in places.
Throughout, we work in a Markov category $\catC$ with conditionals and consider a hidden Markov model $p \colon \monUnit \to \left(Y \otimes X\right)_{\left[n\right]}$ as in \cref{defn:HiddenMarkovModel}.

\subsection{The Bayes smoother}

\begin{defn}\label{defn:Smoother}
	The \newterm{Bayes smoother} at time $t$ is the conditional 
	\begin{equation}
		\label{eqn:BayesSmootherDefn}
		\input{SmootherDef.tikz}
	\end{equation}
\end{defn}

Working directly with this definition has similar practical limitations as with the Bayes filter: the size of the space that one conditions on grows exponentially with the number of observations $n$.
Therefore, we again develop recursive formulas that make it easier to work with.

\subsubsection{The forward-backward algorithm}

Constructing the entire joint distribution $p$ only to marginalize over all but one of the hidden state spaces is wasteful.
Using the structure of the hidden Markov model, we will see that this can indeed be avoided by factoring the relevant marginal into two separate parts and computing each part separately, roughly like so:
\begin{equation}
	\label{eqn:smoother-factorization}
	\input{smoother-chain.tikz}
\end{equation}
We will prove recursive formulas for each of the two parts.
This yields the \newterm{forward algorithm} for the left part and the \newterm{backward algorithm} for the right part.
These morphisms are then composed and conditioned on $Y_0, \dots, Y_n$, a procedure called \newterm{forward-backward algorithm}.
While these procedures are often formulated concretely in $\finstoch$~\cite{rabiner1989tutorial}, our categorical generalization immediately specializes also to a measure-theoretic formulation in $\borelstoch$ and to a possibilistic formulation in $\finsetmulti$.
% Moreover, the categorical formulation is so simple that it will make the algorithm seem almost trivial.
We start with the forward algorithm.

\begin{notn}
	We write $\alpha_t \colon \monUnit \to Y_{\left[t\right]} \otimes X_t$ for the marginal given by
	\begin{equation}
		\label{eqn:naive-forward}
		\input{naive-forward.tikz}
	\end{equation}
\end{notn}

\begin{rmk}
	By definition, the Bayes filter $\filtB{t}$ is the conditional of $\alpha_t$ on $Y_{\left[t\right]}$.
\end{rmk}

The na\"{\i}ve approach to computing $\alpha_t$ in practice would first calculate the entire $p_{\left[t\right]} \colon \monUnit \to \left(Y \otimes X\right)_{\left[t\right]}$ and then marginalize out $X_{\left[t-1\right]}$.
But already in $\finstoch$, this would involve many redundant calculations, resulting in a computation that gets intractable in practice.
The following more workable recursion formula is clear from the definition.

\begin{lem}[The forward algorithm]\label{obs:forward-algorithm}
	The morphisms $\alpha_t$ satisfy the recursion formula 
	\begin{equation}
		\label{eqn:forward-recursion}
		\input{forward-recursion.tikz}
	\end{equation}
	for all $t \in \left[n\right]$, starting with $\alpha_{-1} \coloneqq \id_{\monUnit}$.
\end{lem}

Similarly, the goal of the backward algorithm is to compute the right part of~\eqref{eqn:smoother-factorization}, traditionally denoted $\beta$~\cite[Section~III]{rabiner1989tutorial}.

\begin{notn}
	For a hidden Markov model $p$ and $t \in \left[n\right]$, we write $\beta_t \colon X_t \to Y_{(t,n]}$ for 
	\begin{equation}
		\label{eqn:naive-backward}
		\input{naive-backward.tikz}
	\end{equation}
\end{notn}

Again the following recursion formula is clear from the definition.

\begin{lem}[The backward algorithm]\label{obs:backward-algorithm}
	The morphisms $\beta_t$ satisfy the backwards recursion formula 
	\begin{equation}
		\label{eqn:backward-recursion}
		\input{backward-recursion.tikz}
	\end{equation}
	for all $t \in \left[n\right]$, starting with $\beta_{n+1} \coloneqq \id_{\monUnit}$ where $X_{n+1} \coloneqq Y_{n+1} \coloneqq \monUnit$.
\end{lem}

%More explicitly, the first step of the recursion constructs $\beta_n$ as
%\[
%	\tikzfig{backward-base}
%\]
%which can alternatively be taken to define the base case of the recursion.

%\begin{proof}
%	This simple proof is analogous to the proof of \Cref{obs:forward-algorithm}.
%	We rewrite the right-hand side of~\eqref{eqn:naive-backward} using co-unitality of copy,
%	\begin{equation}
%		\label{eqn:backward-rewrite}
%		\tikzfig{backward-rewrite}
%	\end{equation}
%	where we keep the right-most branch with the deletion on $X_n$ since that corresponds to $\beta_n$.
%	The self-similarity of this string diagram then yields the claimed recursion.
%\end{proof}

\begin{rmk}
	\label{rmk:inconsistencies}
	There are some differences in the literature concerning the definitions of the forward and the forward-backward algorithm, and in particular on whether some conditioning on observations is already considered part of the algorithm or not.
	Common to most prior treatments is that they take place in $\finstoch$ only.\footnote{One exception is Murphy~\cite[Section 18.3.2.3]{murphy2012machine}, who does discuss an equivalent for continuous variables with density.}
	Otherwise accounts differ in the details.
	For example, Rabiner~\cite{rabiner1989tutorial} defines $\alpha_t$ as we do, and conditioning on observations is only done in subsequent steps, such as trying to estimate the most likely hidden state (known as the \emph{Baum--Welch algorithm}).
    Murphy~\cite[Sections 17.4.2, 17.4.3]{murphy2012machine} instead defines $\alpha_t$ to coincide with our Bayes filter $\filtB{t}$, and his description of the forward algorithm is more akin to our instantiated Bayes filter in $\finstoch$, while his version of the forward-backward algorithm corresponds to an instantiated version of our \eqref{eqn:smoother-filter} below in $\finstoch$.\footnote{See \Cref{subsec:instantiated-smoother}.}

	We find it most clear to define $\alpha_t$ and $\beta_t$ as quantities that are not conditioned, and we reserve notation such as $\filtB{t}$ and $\smoothB{n}{t}$ for morphisms that have been conditioned on observations and thus perform Bayesian inference.
\end{rmk}

\begin{lem}
	\label{lem:smoother-filter}
	The Bayes smoother can also be expressed as
	\begin{equation}
		\label{eqn:smoother-filter}
		\input{forward-backward.tikz} \quad =_{\as{\Ch{Y}{t}}} \; \input{filter-backward.tikz}
	\end{equation}
\end{lem}
\begin{proof}
	We have
	\begin{equation}
		\label{eqn:smoother-filter-proof}
		\input{forward-backward.tikz} \;\; \input{filter-backward-proof.tikz}
	\end{equation}
	where the first equation follows by plugging in definitions and the second by \Cref{prop:ConditioningCoherence}.
	Now the inner dashed box is precisely $\filtB{t}$ by definition.
\end{proof}

The conditioning on $Y_{(t,n]}$ in~\eqref{eqn:smoother-filter} may still pose challenges if $n - t$ is not small.
A solution to this has been explored in $\gauss$ called the \emph{two-filter smoother}~\cite{kitagawa1994two,Sarkka}, but we instead explore an alternative smoothing algorithm next.

\subsubsection{Fixed-interval smoothing}
\label{ssec:fixed-interval-smoothing}

Here, we develop the categorical version of \newterm{fixed-interval smoothing}, an existing recursion formula for the Bayes smoothers $\smoothB{n}{t}$ in terms of the Bayes filters $\filtB{t}$~\cite[Section 8.6]{bar2004estimation}.
Its name refers to the time interval for the observation sequence being fixed.
As with the recursion formulas for the Bayes filter from \Cref{sec:BayesFilter}, the complexity reduces to linear in $n$ and $t$.

\begin{prop}[Fixed-interval smoother]
	\label{prop:fixed-interval}
	The following backwards recursion formula for $\smoothB{n}{t}$ holds true for all $t\in \left[0,n\right]$:
	\begin{equation}
		\label{eqn:fixed-interval-step}
		\input{fixed-interval-step.tikz}
	\end{equation}
	where $X_{n+1} = \monUnit$, meaning that $\smoothB{n}{n+1} = \discard_{Y_{\left[n\right]}}$ and $f_{n+1} = \discard_{X_n}$.
\end{prop}

More explicitly, the initial $t = n$ step of the recursion constructs $\smoothB{n}{n} = \filtB{n}$, which can alternatively be taken to be the start of the recursion.

\begin{proof}
	The initial condition $t = n$ holds as $\smoothB{n}{n} = \filtB{n}$ by definition.
	For $t<n$, we combine \Cref{lem:smoother-filter} and the backwards procedure from \Cref{obs:backward-algorithm} to get
	\[
		\input{fixed-interval-backward-rewrite.tikz} \: =_{\as{\Ch{Y}{n}}} \; \input{fixed-interval-cond-rewrite.tikz}
	\]
	where the second step holds by the definition of conditionals.
	Applying \Cref{prop:ConditioningCoherence} together with the fact that copy is deterministic, we can further evaluate this to
	\[
		\hspace{-4mm} \input{fixed-interval-filt-simplified.tikz} \hspace{-2mm} =_{\as{\Ch{Y}{n}}} \hspace{-2mm} \input{fixed-interval-filt-expanded.tikz} \hspace{-2mm} =_{\as{\Ch{Y}{n}}} \input{fixed-interval-next-time-rewrite.tikz}
	\]
	and the bottom right subdiagram is $\smoothB{n}{t+1}$ by \Cref{lem:smoother-filter} again.
\end{proof}

\subsection{The instantiated Bayes smoother}\label{subsec:instantiated-smoother}

As with the Bayes filter, working with the uninstantiated smoother $\smoothB{n}{t}$ can be difficult already because its domain grows exponentially with $n$.
Here, we introduce the \newterm{instantiated Bayes smoother} by specializing the Bayes smoother to a \emph{fixed} sequence of observations, much in the same fashion as we had done for the Bayes filter in \Cref{sec:instantiatedBayesFilter}.

\begin{defn}
	Given a sequence of deterministic morphisms
	\[
		y_0 \colon \monUnit \to Y_0, \qquad y_1 \colon \monUnit \to Y_1, \qquad \ldots, \qquad y_n \colon \monUnit \to Y_n
	\]
	which we abbreviate by $\mathbf{y}_{\left[n\right]}$, the \newterm{instantiated Bayes smoother} is
	\[
		\input{instantBayesSmoother.tikz}
	\]
\end{defn}

From here we can easily construct an instantiated version of the fixed-interval smoothing recursion, which in $\gauss$ recovers the Rauch-Tung-Striebel smoother (\Cref{ssec:rts-smoother}).

\begin{prop}[Instantiated fixed-interval smoother]
	\label{prop:instantiated-fixed-interval}
	The instantiated Bayes smoother can be computed through the following backwards recursion:
	\begin{equation}
		\label{eqn:instantiated-fixed-interval-step}
		\input{instantiated-fixed-interval-step.tikz}
	\end{equation}
	for all $t\in [n]$, where the recursion starts at $\smoothIB{n}{n}{y} = \filtIB{n}\left(\mathbf{y}_{\left[n\right]}\right)$.
\end{prop}

\begin{proof}
	Straightforward from \Cref{prop:fixed-interval} and \cref{prop:ConditioningCoherence}\ref{lem:det_conditional}.
\end{proof}

\subsection{Examples of the instantiated Bayes smoother}

We briefly consider what form the forward-backward algorithm and the fixed-interval smoother take in particular Markov categories, limiting our presentation to $\finstoch$ for the former and $\gauss$ for the latter.

\subsubsection{The forward-backward algorithm in \texorpdfstring{$\finstoch$}{FinStoch}}
\label{ssec:ForwardBackwardFinstoch}

Usually taken to start at $t = 0$, the forward procedure is:

\begin{enumerate}
	\item Initialize as $\alpha_0\left(y_0, x_0\right) \coloneqq g_0\left(y_0 \,|\, x_0\right) \, f_0\left(x_0\right)$.
	\item For $t = 1, \ldots, n$, compute
		\begin{equation}
			\alpha_t\left(\mathbf{y}_{\left[t\right]}, x_t\right) = g_t\left(y_t \,|\, x_t\right) \sum_{x_{t-1}} f_t\left(x_t \,|\, x_{t-1}\right) \, \alpha_{t-1}\left(\mathbf{y}_{\left[t-1\right]}, x_{t-1}\right).
		\end{equation}
\end{enumerate}
Indeed this is the straightforward translation of~\eqref{eqn:forward-recursion} into $\finstoch$.
Similarly, by~\eqref{eqn:backward-recursion} the backward procedure is:
\begin{enumerate}
	\item Initialize as $\beta_n\left(x_n\right) = 1\quad \forall x_n\in X_n$.\footnote{This $\beta_n \colon X_n \to I$ is sometimes described as a ``uniform distribution'', which is technically correct as the only distribution on the singleton set $\monUnit$ is uniform, but we stress that this is not a uniform distribution on $X_n$, already because $X_n$ is the \emph{domain} rather than the codomain of $\beta_n$.}
	\item For $t = n-1, \dots, 0$, compute
		\[
			\beta_t\left(\mathbf{y}_{(t,n]} \,|\, x_t\right) = \sum_{x_{t+1}} g_{t+1}\left(y_{t+1} \,|\, x_{t+1}\right) \, \beta_{t+1}\left(\mathbf{y}_{(t+1,n]} \,|\, x_{t+1}\right) \, f_{t+1}\left(x_{t+1} | x_{t}\right).
		\]
\end{enumerate}

Finally, performing the conditioning with the observed values $\mathbf{y}_{[n]}$ plugged in to~\eqref{eqn:smoother-filter-proof} produces the desired result,
\begin{equation}
	\label{eq:forward-backward-estimate}
	\smoothIB{n}{t}{y}\left(x_t\right) =
	\frac{\alpha_t\left(\mathbf{y}_{\left[t\right]}, x_t\right)\beta_t\left(\mathbf{y}_{(t,n]} \,|\, x_t\right)}
	{\sum_{x_t'}\alpha_t\left(\mathbf{y}_{\left[t\right]}, x_t'\right)\beta_t\left(\mathbf{y}_{(t,n]} \,|\, x_t'\right)}
\end{equation}
We leave the instantiation of this algorithm in other Markov categories, like $\borelstoch$, $\gauss$, $\finsetmulti$ and parametric Markov categories, to the reader.

\subsubsection{The RTS smoother as the instantiated Bayesian smoother}
\label{ssec:rts-smoother}

Shortly after the Kalman filter, a Gaussian fixed-interval backwards recursion was developed~\cite{rts-smoother}, now commonly known as the \newterm{Rauch-Tung-Striebel (RTS) smoother} or \newterm{Kalman smoother}~\cite[Section 18.3.2]{murphy2012machine}.
Similar to the Kalman filter in \Cref{ssec:KalmanFilter}, we show that our categorical fixed-interval smoother specializes to the RTS smoother in $\gauss$, slightly generalized to allow for biased noise.

Keeping notation as in the Kalman filter case from \Cref{ssec:KalmanFilter}, we additionally write the parameters for the filter output as
\[
	\filtIB{t}\left(\mathbf{y}_{\left[t\right]}\right) = \mathcal{N}\left(m_t, P_t\right),
\]
noting that we assume these quantities to have been computed already e.g.~through the Kalman filter.
To indicate the relation, let us denote the output parameters of the instantiated smoother with ``hats,'' meaning that
\[
	\smoothIB{n}{t}{y} = \mathcal{N}\left(\hat{m}_t, \hat{P}_t\right).
\]
To determine the recursion formulas for these new parameters, we first write out the joint state inside the dashed box in~\eqref{eqn:instantiated-fixed-interval-step} as
\[
	\begin{pmatrix}
		x_t \\ \tilde{x}_{t+1} 
	\end{pmatrix} = \mathcal{N}(a_t, S_t),
\]
where the parameters
\[
	a_t = \begin{pmatrix}
		m_t \\ \tilde{m}_{t+1} 
		\end{pmatrix} \qquad \text{ and }
		\qquad
	S_t = \begin{pmatrix}
		P_t & P_t A_{t+1}^t \\
		A_{t+1}P_t & \tilde{P}_{t+1} 
		\end{pmatrix}
\]
are determined by
\begin{equation}
	\label{eqn:instantiated-fixed-interval-step-gauss}
	\tilde{m}_{t+1} = A_{t+1} m_{t} + v_{t+1}, \qquad \tilde{P}_{t+1} = A_{t+1} P_{t} A_{t+1}^t + Q_{t+1},
\end{equation}
as in \eqref{eq:gauss_prediction} for the Kalman filter.
The conditional of this joint state is represented by the equation
\[
	\hat{x}_t = C_t \hat{x}_{t+1} + \mathcal{N}\left(\ m_t - C_t \tilde{m}_{t+1}, \ P_t - C_t \tilde{P}_{t+1} C_t^t \right),
\]
where we have introduced the \newterm{optimal smoother gain}
\[
	C_t \coloneqq P_t A_{t+1}^t \tilde{P}_{t+1}^{-}.
\]
Composing this with $\smoothIB{n}{t+1}{y}$, we obtain the recursion formulas that define the RTS smoother,
\begin{equation}
	\label{eq:rts_smoother_result}
	\hat{m}_t = m_t + C_t \left(\hat{m}_{t+1} - \tilde{m}_{t+1}\right), \qquad 
	\hat{P}_t = P_t + C_t \left(\hat{P}_{t+1} - \tilde{P}_{t+1}\right) C_t^t,
\end{equation}
where the tilde quantities are given by~\eqref{eqn:instantiated-fixed-interval-step-gauss}.

\bibliographystyle{abbrvnat}
\bibliography{hmm}

\end{document}

%% file: markov2.tikzstyles
\tikzstyle{morphism}=[fill=white, draw=black, shape=rectangle]
\tikzstyle{medium box}=[fill=white, draw=black, shape=rectangle, minimum width=0.7cm, minimum height=0.7cm]
\tikzstyle{large morphism}=[fill=white, draw=black, shape=rectangle, minimum width=1.9cm, minimum height=0.9cm]
\tikzstyle{bn}=[fill=black, draw=black, shape=circle, inner sep=1.5pt]
\tikzstyle{state}=[fill=white, draw=black, regular polygon, regular polygon sides=3, minimum width=0.8cm, shape border rotate=180, inner sep=0pt]
\tikzstyle{medium state}=[fill=white, draw=black, regular polygon, regular polygon sides=3, minimum width=1.3cm, inner sep=0pt, shape border rotate=180]
\tikzstyle{large state}=[fill=white, draw=black, regular polygon, regular polygon sides=3, minimum width=2.2cm, shape border rotate=180, inner sep=0pt]
\tikzstyle{wide state}=[fill=white, draw=black, shape=isosceles triangle, minimum width=0.8cm, shape border rotate=270, inner sep=1.4pt, minimum height=0.5cm, isosceles triangle apex angle=80]
\tikzstyle{very wide state}=[fill=white, draw=black, shape=isosceles triangle, isosceles triangle stretches, minimum width=3cm, shape border rotate=270, inner sep=0pt]
\tikzstyle{wn}=[fill=white, draw=black, shape=circle, inner sep=1.5pt]
\tikzstyle{blue morphism}=[fill=white, draw={rgb,255: red,15; green,0; blue,150}, shape=rectangle, text={rgb,255: red,15; green,0; blue,150}, tikzit category=blue]
\tikzstyle{blue state}=[fill=white, draw={rgb,255: red,15; green,0; blue,150}, shape=circle, regular polygon, regular polygon sides=3, minimum width=0.8cm, shape border rotate=180, inner sep=0pt, text={rgb,255: red,15; green,0; blue,150}, tikzit category=blue]
\tikzstyle{blue node}=[fill={rgb,255: red,15; green,0; blue,150}, draw={rgb,255: red,15; green,0; blue,150}, shape=circle, tikzit category=blue, inner sep=1.5pt]
\tikzstyle{blue}=[text={rgb,255: red,15; green,0; blue,150}, tikzit draw={rgb,255: red,191; green,191; blue,191}, tikzit category=blue, tikzit fill=white, inner sep=0mm]
\tikzstyle{blue wide state}=[fill=white, draw={rgb,255: red,15; green,0; blue,150}, text={rgb,255: red,15; green,0; blue,150}, shape=isosceles triangle, minimum width=0.8cm, shape border rotate=270, inner sep=1.4pt, minimum height=0.5cm, isosceles triangle apex angle=80]
\tikzstyle{red node}=[fill={rgb,255: red,150; green,0; blue,2}, draw={rgb,255: red,150; green,0; blue,2}, shape=circle, inner sep=1.5pt]
\tikzstyle{Purple node}=[fill={rgb,255: red,120; green,0; blue,120}, draw={rgb,255: red,120; green,0; blue,120}, text={rgb,255: red,120; green,0; blue,120}, shape=circle, inner sep=1.5pt]
\tikzstyle{red}=[text={rgb,255: red,150; green,0; blue,2}, inner sep=0mm, tikzit fill=white, tikzit draw={rgb,255: red,191; green,191; blue,191}]
\tikzstyle{purple}=[text={rgb,255: red,150; green,0; blue,150}, inner sep=0mm, tikzit fill=white, tikzit draw={rgb,255: red,191; green,191; blue,191}]
\tikzstyle{white morphism}=[fill=white, draw=white, shape=rectangle, tikzit draw={rgb,255: red,139; green,139; blue,139}]
\tikzstyle{leak morphism}=[fill=white, draw={rgb,255: red,120; green,0; blue,85}, shape=rectangle, text={rgb,255: red,120; green,0; blue,85}, tikzit category=leak]
\tikzstyle{leak}=[text={rgb,255: red,120; green,0; blue,85}, inner sep=0mm, tikzit fill=white, tikzit draw={rgb,255: red,191; green,191; blue,191}, tikzit category=leak]
\tikzstyle{leak node}=[fill={rgb,255: red,120; green,0; blue,85}, draw={rgb,255: red,120; green,0; blue,85}, shape=circle, inner sep=1.5pt, tikzit category=leak]
\tikzstyle{curly brace}=[decorate, decoration=brace]

\tikzstyle{arrow}=[->]
\tikzstyle{dashed box}=[-, dashed]
\tikzstyle{blue line}=[-, draw={rgb,255: red,15; green,0; blue,150}, tikzit category=blue]
\tikzstyle{red arrow}=[-, draw={rgb,255: red,150; green,0; blue,2}, tikzit category=red]
\tikzstyle{purple line}=[draw={rgb,255: red,120; green,0; blue,120}, >=stealth, shorten <=2pt, shorten >=2pt, -]
\tikzstyle{protected purple line}=[draw={rgb,255: red,120; green,0; blue,120}, >=stealth, shorten <=2pt, shorten >=2pt, preaction={line width=1.8pt, white, draw}, -]
\tikzstyle{mapsto}=[{|->}]
\tikzstyle{double wire}=[-, double]
\tikzstyle{protected}=[-, preaction={line width=1.8pt,white,draw}]
\tikzstyle{leak arrow}=[-, line join=round, decorate, decoration={snake, segment length=4, amplitude=0.75, pre=curveto, post=curveto, pre length=1pt, post length=1pt}]
\tikzstyle{protected leak arrow}=[-, line join=round, decorate, decoration={snake, segment length=4, amplitude=0.75, pre=curveto, post=curveto, pre length=1pt, post length=1pt}, preaction={line width=1.8pt, white, draw}]
\tikzstyle{hollow arrow}=[-, very thin, white, preaction={line width=0.7pt,draw={rgb,255: red,120; green,0; blue,85}}, tikzit category=leak, tikzit draw={rgb,255: red,150; green,0; blue,120}]
\tikzstyle{protected hollow arrow}=[-, very thin, white, preaction={line width=0.7pt,draw={rgb,255: red,120; green,0; blue,85},preaction={line width=2.1pt,white,draw}}, tikzit category=leak, tikzit draw={rgb,255: red,150; green,0; blue,120}]

%% file: figures/del.tikz
\begin{tikzpicture}
	\begin{pgfonlayer}{nodelayer}
		\node [style=bn] (0) at (0, 1) {};
		\node [style=none] (1) at (0, -1) {};
		\node [style=none] (2) at (0, -1.5) {$X$};
	\end{pgfonlayer}
	\begin{pgfonlayer}{edgelayer}
		\draw (0) to (1.center);
	\end{pgfonlayer}
\end{tikzpicture}

%% file: figures/single_output_probability.tikz
\begin{tikzpicture}
	\begin{pgfonlayer}{nodelayer}
		\node [style=none] (1) at (0, 1.25) {};
		\node [style=none] (4) at (0, 1.75) {$X$};
		\node [style=state] (5) at (0, -0.75) {$\, p \, $};
	\end{pgfonlayer}
	\begin{pgfonlayer}{edgelayer}
		\draw (1.center) to (5);
	\end{pgfonlayer}
\end{tikzpicture}